\documentclass{article}
\usepackage[final]{neurips_2023}

\title{Entropy-dissipation Informed Neural Network\\ for McKean-Vlasov Type PDEs}
\usepackage{times}

\author{
	Zebang Shen\thanks{Authors are listed in alphabetic order.}\\
	ETH Z\"urich \\
	\texttt{zebang.shen@inf.ethz.ch} \\
	\And
	Zhenfu Wang$^*$ \\
	Peking University \\
	\texttt{zwang@bicmr.pku.edu.cn}
}
\usepackage{amssymb,amsmath,amsthm}
\usepackage{color}
\usepackage{makecell}
\usepackage{enumitem}
\usepackage{natbib}
\usepackage{graphicx}
\usepackage{hyperref}       
\usepackage{MnSymbol}

\usepackage{amsmath,amsfonts,bm}

\def\eqref#1{equation~\ref{#1}}

\def\1{\bm{1}}

\def\rmB{{\mathbf{B}}}

\def\rmX{{\mathbf{X}}}

\def\vmu{{\bm{\mu}}}

\def\va{{\bm{a}}}
\def\vb{{\bm{b}}}

\def\vf{{\bm{f}}}
\def\vg{{\bm{g}}}

\def\vs{{\bm{s}}}

\def\vu{{\bm{u}}}
\def\vv{{\bm{v}}}

\def\vx{{\bm{x}}}
\def\vy{{\bm{y}}}
\def\vz{{\bm{z}}}

\def\mI{{\bm{I}}}

\def\mSigma{{\bm{\Sigma}}}

\DeclareMathAlphabet{\mathsfit}{\encodingdefault}{\sfdefault}{m}{sl}
\SetMathAlphabet{\mathsfit}{bold}{\encodingdefault}{\sfdefault}{bx}{n}

\def\gA{{\mathcal{A}}}
\def\gB{{\mathcal{B}}}
\def\gC{{\mathcal{C}}}
\def\gD{{\mathcal{D}}}

\def\gF{{\mathcal{F}}}

\def\gJ{{\mathcal{J}}}

\def\gL{{\mathcal{L}}}

\def\gX{{\mathcal{X}}}

\def\sA{{\mathbb{A}}}

\def\sR{{\mathbb{R}}}

\def\sX{{\mathbb{X}}}

\newcommand{\E}{\mathbb{E}}

\newcommand{\KL}{D_{\mathrm{KL}}}

\newtheorem{theorem}{Theorem}
\newtheorem{proposition}{Proposition}
\newtheorem{assumption}{Assumption}
\newtheorem{remark}{Remark}
\newtheorem{lemma}{Lemma}
\newcommand{\Law}{\mathrm{Law}}

\newcommand{\ud}{\mathrm{d}}
\newcommand{\X}{\mathcal{X}}
\newcommand{\udiv}{\, \mathrm{div}}
\newcommand{\Uniform}{\mathrm{Uniform}}
\let\KL\relax
\newcommand{\KL}{\mathbf{KL}}
\newcommand{\EINN}{\texttt{EINN}}
\newcommand{\defi}{\overset{\operatorname{def}}{=}}
\newcommand{\Lip}{\mathrm{Lip}}

\begin{document}

\maketitle

\begin{abstract}%
    The McKean-Vlasov equation (MVE) describes the collective behavior of particles subject to drift, diffusion, and mean-field interaction. In physical systems, the interaction term can be singular, i.e. it diverges when two particles collide. Notable examples of such interactions include the Coulomb interaction, fundamental in plasma physics, and the Biot-Savart interaction, present in the vorticity formulation of the 2D Navier-Stokes equation (NSE) in fluid dynamics.
    Solving MVEs that involve singular interaction kernels presents a significant challenge, especially when aiming to provide rigorous theoretical guarantees. In this work, we propose a novel approach based on the concept of entropy dissipation in the underlying system. We derive a potential function that effectively controls the KL divergence between a hypothesis solution and the ground truth.
    Building upon this theoretical foundation, we introduce the Entropy-dissipation Informed Neural Network (\EINN) framework for solving MVEs. In \EINN, we utilize neural networks (NN) to approximate the underlying velocity field and minimize the proposed potential function. By leveraging the expressive power of NNs, our approach offers a promising avenue for tackling the complexities associated with singular interactions.
    To assess the empirical performance of our method, we compare \EINN\ with SOTA NN-based MVE solvers. The results demonstrate the effectiveness of our approach in solving MVEs across various example problems.
\end{abstract}

\vspace{-.4cm}
\section{Introduction}
\vspace{-.2cm}
Scientists use Partial Differential Equations (PDEs) to describe natural laws and predict the dynamics of real-world systems. As PDEs are of fundamental importance, a growing area in machine learning is the use of neural networks (NN) to solve these equations {\citep{han2018solving,zhang2018deep,raissi2020hidden,cai2021physics,karniadakis2021physics,cuomo2022scientific}}.
An important category of PDEs is the McKean-Vlasov equation (MVE), which models the dynamics of a stochastic particle system with mean-field interactions
\begin{equation} \label{eqn_MVE_particle}
		\ud \rmX_t = - \nabla V(\rmX_t) \ud t  + K \ast \bar \rho_t(\rmX_t) \ud t  + \sqrt{2\nu}  \ud \rmB_t,  \quad
		\bar \rho_t = \Law(\rmX_t).  
\end{equation}
Here $\rmX_t \in \X$ denotes a random particle' position, $\X$ is either $\mathbb{R}^d$ or the torus $\Pi^d$ (a cube $[-L, L]^d$ with periodic boundary condition), $V: \sR^d\rightarrow\sR$ denotes a {\em known} potential, $K: \sR^d \rightarrow \sR^d$ denotes some interaction kernel and the convolution operation is defined as $h \ast \phi = \int_{\X} h(\vx-\vy) \phi(\vy) d \vy$, $\{\rmB_t\}_{t\geq 0}$ is the standard  $d$-dimensional Wiener process with $\nu\geq 0$ being the diffusion coefficient, and $\bar \rho_t:\X \rightarrow \sR$ is the law or the probability density function of the random variable $\rmX_t$ and the initial data $\bar \rho_0$ is given.
Under mild regularity conditions, the density function $\bar \rho_t$ satisfies the MVE
\begin{equation} \label{eqn_MVE}
    \text{(MVE)}\quad 
	\partial_t \bar \rho_t(\vx) + \udiv \left( \bar \rho_t (- \nabla V(\vx) + K \ast \bar \rho_t(\vx))\right) = \nu \Delta  \bar \rho_t(\vx),
\end{equation}
where $\udiv$ denotes the divergence operator, $\udiv\ h(\vx) = \sum_{i=1}^d {\partial h_i}/{\partial x_i}$ for a velocity field $h:\sR^d\rightarrow\sR^d$, $\Delta$ denotes the Laplacian operator defined as $\Delta \phi = \udiv( \nabla \phi)$,  where $\nabla \phi$ denotes the gradient of a scalar function $\phi:\sR^d\rightarrow\sR$. Note that all these operators are applied only on the spatial variable $\vx$.

In order to describe dynamics in real-world phenomena such as electromagnetism \citep{golse2016dynamics} and fluid mechanics {\citep{majda2002vorticity}}, the interaction kernels $K$ in the MVE can be highly \emph{singular}, i.e. $\|K(\vx)\| \rightarrow \infty$ when  $\|\vx\|\rightarrow 0$.
Two of the most notable examples are the Coulomb interactions 
\begin{equation} \label{eqn_coulomb_interaction}
    \text{(Coulomb Kernel)}\quad 
	K(\vx) = -\nabla g(\vx),\ \mathrm{with}\ g(\vx) = \begin{cases}
		\left((d-2)S_{d-1}(1)\right)^{-1}\|\vx\|^{-(d-2)}, & d\geq 3,\\
		-(2\pi)^{-1}\log \|\vx\|,  & d = 2,
	\end{cases}
\end{equation}
with $S_{d-1}(1)$ denoting the surface area of the unit sphere in $\sR^d$,
and the vorticity formulation of the 2D Navier-Stokes equation (NSE) where the interaction kernel $K$ is given by the Biot-Savart law
\begin{equation} \label{eqn_nse}
    \text{(Biot-Savart Kernel)}\quad 
	K(\vx) = \frac{1}{2\pi} \frac{\vx^\perp}{\|\vx\|^2} = \frac{1}{2\pi} \left(-\frac{x_2}{\|\vx\|^2}, \frac{x_1}{\|\vx\|^2}\right), 
\end{equation}
where $\vx= (x_1, x_2)$ and  $\|\vx\|$ denotes the Euclidean norm of a vector.\vspace{1mm}\\
Classical methods for solving MVEs, including finite difference, finite volume, finite element, spectral methods, and particle methods, have been developed over time. A common drawback of these methods lies in the constraints of their solution representation: Sparse representations, such as less granular grids, cells, meshes, fewer basis functions, or particles, may lead to an inferior solution accuracy; On the other hand, dense representations incur higher computational and memory costs. \vspace{1mm}\\
As a potent tool for function approximation, NNs are anticipated to overcome these hurdles and handle higher-dimensional, less regular, and more complex systems efficiently \citep{weinan2021algorithms}.
The most renowned NN-based algorithm is the Physics Informed Neural Network (PINN) \citep{raissi2019physics}.
The philosophy behind the PINN method is that solving a PDE system is equivalent to finding the root of the corresponding differential operators. PINN tackles the latter problem by directly parameterizing the hypothesis solution with an NN and training it to minimize the $\gL^2$ functional residual of the operators. 
As a versatile PDE solver, PINN may fail to exploit the underlying dynamics of the PDE, which possibly leads to inferior performance on task-specific solvers. For example, on the 2D NSE problem, a recent NN-based development \citet{zhang2022drvn} surpasses PINN and sets a new SOTA empirical performance, which however lacks rigorous theoretical substantiation.
Despite the widespread applications of PINN, rigorous error estimation guarantees are scarce in the literature. While we could not find results on the MVE with the Coulomb interaction, only in a very recent paper \citep{deerror}, the authors establish for NSE that the PINN loss controls the discrepancy between a candidate solution and the ground truth. 
We highlight that their result holds \emph{average-in-time}, meaning that at a particular timestamp $t \in [0, T]$, a candidate solution with small PINN loss may still significantly differ from the true solution. In contrast, all guarantees in this paper are \emph{uniform-in-time}. Moreover, there is a factor in the aforementioned guarantee that \emph{exponentially depends} on the total evolving time $T$, while the factor in our guarantee for the NSE is \emph{independent} of $T$.
We highlight that these novel improvements are achieved for the proposed \EINN\ framework since we take a completely different route from PINN: Our approach is explicitly designed to exploit the underlying dynamics of the system, as elaborated below.

\vspace{-.3cm}
\paragraph{Our approach} Define the operator
\vspace{-.1cm}
\begin{equation} \label{eqn_operator_A}
    \mathcal{A}[\rho] \defi - \nabla V + K \ast \rho -  \nu \nabla \log \rho.
\end{equation}
By noting $\Delta  \bar \rho_t = \udiv (\bar \rho_t \nabla \log \bar  \rho_t)$, we can rewrite the MVE in the form of a continuity equation
\begin{equation} \label{eqn_MVE_CE}
	\partial_t \bar  \rho_t(\vx)  + \udiv \Big( \bar \rho_t(\vx) \gA[\bar \rho_t](\vx)\Big) = 0.
\end{equation}
For simplicity, we will refer to $\mathcal{A}[\bar \rho_t]$ as the \emph{underlying velocity}.
Consider another time-varying \emph{hypothesis velocity} field $f:\sR\times \sR^d\rightarrow\sR$ and let $\rho^f_t$ be the solution to the continuity equation
\vspace{-1mm}
\begin{equation} \label{eqn_CE}
    \text{(hypothesis solution)}\quad 
	\partial_t \rho^f_t(\vx)  + \udiv (\rho_t^f(\vx) f(t, \vx) ) = 0,\ \rho^f_0 = \bar \rho_0
\end{equation}
for $t \in [0, T]$, where we recall that the initial law  $\bar \rho_0$ is known.
We will refer to $\rho_t^f$ as the \emph{hypothesis solution} and use the superscript to emphasize its dependence on the hypothesis velocity field $f$.
We propose an Entropy-dissipation Informed Neural Network framework (\EINN), which trains an NN parameterized hypothesis velocity field $f_\theta$ by minimizing the following \EINN\ loss
\begin{equation} \label{eqn_self_consistency_potential}
\text{(\EINN\ loss)}\quad 
        R(f_\theta) \defi \int_0^T \int_\X \|f_\theta(t,\vx) - \mathcal{A}[\rho_t^{f_\theta}](\vx) \|^2 {\rho_t^{f_\theta}}(\vx)\ud \vx \ud t.    
\end{equation}
The objective (\ref{eqn_self_consistency_potential}) is obtained by studying the stability of carefully constructed Lyapunov functions. These Lyapunov functions draw inspiration from the concept of entropy dissipation in the system, leading to the name of our framework.
We highlight that we provide a rigorous error estimation guarantee for our framework for MVEs with singular kernels (\ref{eqn_coulomb_interaction}) and (\ref{eqn_nse}), showing that when $R(f_\theta)$ is sufficiently small, $\rho_t^{f_\theta}$ recovers the ground truth $\bar \rho_t$ in the KL sense, uniform in time.

\begin{theorem}[Informal] \label{thm_informal}
	Suppose that the initial density function $\bar \rho_0$ is sufficiently regular and the hypothesis velocity field  $f_t(\cdot) = f(t, \cdot)$ is at least three times continuously differentiable both in $t$ and $x$. 
	We have for the MVE with a bounded interaction kernel $K$ or with the singular Coulomb (\ref{eqn_coulomb_interaction}) or Biot-Savart (\ref{eqn_nse}) interaction, the KL divergence between the hypothesis solution $\rho_t^f$ and the ground truth $\bar \rho_t$ is controlled by the \EINN\ loss for any $t\in[0, T]$, i.e. there exists some constant $C>0$, 
	\begin{equation}
		\sup_{t \in [0, T]}	\KL(\rho_t^f, \bar \rho_t) \leq C R(f).
	\end{equation} 
\end{theorem}
\vspace{-3mm}
Having stated our main result, we elaborate on the difference between \EINN\ and PINN in terms of information flow over time, which explains why \EINN\ achieves better theoretical guarantees:
In PINN, the residuals at different time stamps are independent of each other and hence there is no information flow from the residual at time $t_1$ to the one at time $t_2 (> t_1)$. In contrast, in the \EINN\ loss (\ref{eqn_self_consistency_potential}), incorrect estimation made in $t_1$ will also affect the error at $t_2$ through the hypothesis solution $\rho_t^f$. Such an information flow gives a stronger gradient signal when we are trying to minimize the \EINN\ loss, compared to the PINN loss. It partially explains why we can obtain the novel uniform-in-time estimation as opposed to the average-in-time estimation for PINN and why the constant $C$ in the NSE case is independent of $T$ for \EINN\ (Theorem \ref{NSMainEstimate}), but exponential in $T$ for PINN. 

\vspace{-1mm}

\textbf{Contributions.} In summary, we present a novel NN-based framework for solving the MVEs. Our method capitalizes on the entropy dissipation property of the underlying system, ensuring robust theoretical guarantees even when dealing with singular interaction kernels. We elaborate on the contributions of our work from theory, algorithm, and empirical perspectives as follows.
\vspace{-2mm}
\begin{enumerate}[wide, labelwidth=!, labelindent=0pt]
\vspace{-1mm}
	\item (Theory-wise) By studying the stability of the MVEs with bounded interaction kernels or with singular interaction kernels in the Coulomb (\ref{eqn_coulomb_interaction}) and the Biot-Savart case (\ref{eqn_nse}) (the 2D NSE) via entropy dissipation, we establish the error estimation guarantee for the \EINN\ loss on these equations. Specifically, we design a potential function $R(f)$ of a hypothesis velocity $f$ such that $R(f)$ controls the KL divergence between the hypothesis solution $\rho_t^f$ (defined in equation (\ref{eqn_CE})) and the ground truth solution $\bar \rho_t$ for any time stamp within a given time interval $[0, T]$. 
    A direct consequence of this result is that $R(f)$ can be used to assess the quality of a generic hypothesis solution to the above MVEs and $\rho_t^f$ exactly recovers $\bar \rho_t$ in the KL sense given that $R(f) = 0$.
    \vspace{-1mm}
	\item (Algorithm-wise) When the hypothesis velocity field is parameterized by an NN, i.e. $f = f_\theta$ with $\theta$ being some finite-dimensional parameters, the \EINN\ loss $R(f_\theta)$ can be used as the loss function of the NN parameters $\theta$. We discuss in detail how an estimator of the gradient $\nabla_\theta R(f_\theta)$ can be computed so that stochastic gradient-based optimizers can be utilized to train the NN. In particular, for the 2D NSE (the Biot-Savart case (\ref{eqn_nse})), we show that the singularity in the gradient computation can be removed   
	by exploiting the anti-derivative of the Biot-Savart kernel.
    \vspace{-1mm}
	\item (Empirical-wise) We compare the proposed approach, derived from our novel theoretical guarantees, with SOTA NN-based algorithms for solving the MVE with the Coulomb interaction and the 2D NSE (the Biot-Savart interaction). We pick specific instances of the initial density $\bar \rho_0$, under which explicit solutions are known and can be used as the ground truth to test the quality of the hypothesis ones.
	Using NNs with the same complexity (depth, width, and structure), we observe that the proposed method significantly outperforms the included baselines.
\end{enumerate}

%

\vspace{-4mm}
\section{Entropy-dissipation Informed Neural Network}
\vspace{-3mm}
In this section, we present the proposed \EINN\ framework for the MVE.
To understand the intuition behind our design, we first write the continuity equation (\ref{eqn_CE}) in a similar form as the MVE (\ref{eqn_MVE_CE}):
\vspace{-1mm}
\begin{equation} \label{eqn_CE_as_MVE}
	\partial_t \rho^f_t(\vx) + \udiv \bigg(\rho_t^f(\vx) \Big(\gA[\rho_t^f](\vx) + \delta_t(\vx) \Big) \bigg) = 0,
\end{equation}
where $f$ is the hypothesis velocity (recall that $f_t(\cdot) = f(t, \cdot)$) and 
\begin{equation} \label{eqn_perturbation}
    \text{(Perturbation)}\quad 
	\delta_t(\vx) \defi f_t(\vx) - \gA[\rho_t^f](\vx)
\end{equation}
can be regarded as a perturbation to the original MVE system.
Consequently, it is natural to study the deviation of the hypothesis solution $\rho^f_t$ from the true solution $\bar \rho_t$ using an appropriate Lyapunov function $L(\rho_t^f, \bar \rho_t)$. 
The functional relation between this deviation and the perturbation is termed as the {\emph{stability}} of the underlying dynamical system, which allows us to derive the \EINN\ loss (\ref{eqn_self_consistency_potential}). 
Following this idea, the design of the \EINN\ loss can be determined by the choice of the Lyapunov function $L$ used in the stability analysis. 
In the following, we describe the Lyapunov function used for the MVE with the Coulomb interaction and the 2D NSE (MVE with Biot-Savart interaction). The proof of the following results are the major theoretical contributions of this paper and will be elaborated in the analysis section \ref{section_analysis}.
\begin{itemize}[leftmargin=*]
	\item For the MVE with the Coulomb interaction, we choose $L$ to be the \emph{modulated free energy} $E$ (defined in equation (\ref{DefModFree})) which is originally proposed in \citep{bresch2019mean} to establish the mean-field limit of a corresponding interacting particle system. We have (setting $L = E$)
	{
	\begin{equation}
		\frac{\ud}{\ud t } E(\rho_t^f, \bar \rho_t) \leq \frac 1 2  \int_{\mathcal{X}} \rho_t^f(\vx) \|\delta_t(\vx) \|^2 \ud \vx  + C\, E(\rho_t^f, \bar \rho_t),
	\end{equation}
	}
	where $C$ is a universal constant depending on $\nu $ and $(\bar\rho_t)_{t \in [0, T]}$. 
	\item For the 2D NSE (MVE with the  Biot-Savart interaction), we choose $L$ as the KL divergence. Our analysis is inspired by \citep{jabin2018quantitative} which for the first time establishes the quantitative mean-field limit of the stochastic interacting particle systems where the interaction kernel can be  in some negative Sobolev space. We have 
	{
	\begin{equation}
		\frac{\ud }{\ud  t} \mathbf{KL}(\rho_t^f, \bar \rho_t) \leq - \frac \nu  2 \int_\X \rho^f_t(\vx) \|\nabla \log \frac{\rho^f_t}{\bar \rho_t}(\vx)\|^2 +C\, \mathbf{KL}(\rho_t^f, \bar \rho_t)+ \frac{1}{\nu } \int_\X \rho_t^f(\vx)  \|\delta_t(\vx)\|^2 \ud \vx, 
	\end{equation}
	}
where again $C$ is a universal constant depending on $\nu $ and $(\bar\rho_t)_{t \in [0, T]}$. 
\end{itemize}
After applying Gr\"onwall's inequality on the above results, we can see that the \EINN\ loss (\ref{eqn_self_consistency_potential}) is precisely the term derived by stability analysis of the MVE system with an appropriate Lyapunov function.
In the next section, we elaborate on how a stochastic approximation of $\nabla_\theta R(f_\theta)$ can be efficiently computed for a parameterized hypothesis velocity field $f = f_\theta$ so that stochastic optimization methods can be utilized to minimize $R(f_\theta)$.

\subsection{Stochastic Gradient Computation with Neural Network Parameterization} \label{section_NN_parameterization}
While the choice of the \EINN\ loss (\ref{eqn_self_consistency_potential}) is theoretically justified through the above stability study, in this section, we show that it admits an estimator which can be efficiently computed. 
Define the flow map $X_t$ via the ODE $\ud \vx(t) = f_t(\vx(t); \theta)\ud t$ with $\vx(0) = \vx_0$  such that $\vx(t) = X_t(\vx_0)$.
From the definition of the push-forward measure, one has $\rho^f_t = X_t\sharp\bar\rho_0$.
Recall the definitions of the \EINN\ loss $R(f)$ in equation (\ref{eqn_self_consistency_potential}) and the perturbation $\delta_t$ in equation (\ref{eqn_perturbation}). Use the change of variable formula of the push-forward measure in (a) and Fubini's theorem in (b). We have
\begin{equation} \label{eqn_pathwise_reformulation}
	R(f) = \int_0^T \|\delta_t\|_{\rho_t^f}^2 d t \stackrel{(a)}{=} \int_0^T \|\delta_t\circ X_t\|_{\rho_0}^2 d t \stackrel{(b)}{=} \int \int_0^T \|\delta_t\circ X_t(\vx_0)\|^2 d t d \bar \rho_0(\vx_0).
\end{equation}
Consequently, by defining the trajectory-wise loss (recall $\vx(t) = X_t(\vx_0)$)
\begin{equation} \label{eqn_trajectory_wise_loss}
	R(f; \vx_0) = \int_0^T \|\delta_t\circ X_t(\vx_0)\|^2 d t = \int_0^T \|\delta_t(\vx(t))\|^2 d t,
\end{equation}
we can write the potential function (\ref{eqn_self_consistency_potential}) as an expectation $R(f) = \E_{\vx_0\sim \bar \rho_0}[R(f; \vx_0)]$.
Similarly, when $f$ is parameterized as $f = f_\theta$, we obtain the expectation form $\nabla_\theta R(f_\theta) = \E_{\vx_0\sim\bar \rho_0} [\nabla_\theta R(f_\theta; \vx_0)]$.

We show $\nabla_\theta R(f_\theta; \vx_0)$ can be computed accurately, via the adjoint method (for completeness see the derivation of the adjoint method in appendix \ref{appendix_adjoint_method}). 
As a recap, suppose that we can write $R(f_\theta; \vx_0)$ in a standard ODE-constrained form
$R(f_\theta; \vx_0) = \ell(\theta) = \int_0^T g(t, \vs(t), \theta) d t$,
where $\{\vs(t)\}_{t\in[0, T]}$ is the solution to the ODE $\frac{d }{d t} \vs(t) = \psi(t, \vs(t); \theta)$ with $\vs(0) = \vs_0$, and $\psi$ is a known transition function.
The adjoint method states that the gradient $\frac{d}{d \theta} \ell(\theta)$ can be computed as
\begin{equation}
    \text{(Adjoint Method)}\quad 
	\frac{d \ell}{d \theta} = \int_0^T a(t)^\top\frac{\partial \psi}{\partial \theta}(t, \vs(t); \theta) + \frac{\partial g}{\partial \theta}(t, \vs(t); \theta) \ud t.
\end{equation}
where $a(t)$ solves the final value problems $\frac{d}{d t} a(t)^\top + a(t)^\top \frac{\partial  \psi}{\partial  s}(t, \vs(t); \theta) + \frac{\partial g}{\partial s}(t, \vs(t); \theta) = 0, a(T) = 0$.
In the following, we focus on how $R(f_\theta; \vx_0)$ can be written in the above ODE-constrained form.

\paragraph{Write $R(f_\theta; \vx_0)$ in  ODE-constrained form}
Expanding the definition of $\delta_t$ in equation (\ref{eqn_perturbation}) gives
\begin{equation} \label{eqn_delta_x_t}
	\delta_t(\vx(t)) = f_t(\vx(t)) - \left( - \nabla V(\vx(t)) + K\ast \rho_t^f(\vx(t)) - \nu \nabla \log \rho_t^f(\vx(t))\right).
\end{equation}
Note that in the above quantity, $f$ and $V$ are known functions. 
Moreover, it is known that $\nabla \log \rho_t^f(\vx(t))$ admits a closed form dynamics (e.g. see Proposition 2 in \citep{shen22a})
\begin{equation} \label{eqn_dynamics_of_score}
	\frac{d }{d t}\nabla \log \rho_{t}^{f}(\vx(t)) = - \nabla  \left(\udiv f_{t}(\vx(t); \theta)\right) -   \left(\gJ_{f_{t}}(\vx(t); \theta)\right)^\top \nabla \log \rho_{t}^{f}(\vx(t)),
\end{equation}
which allows it to be explicitly computed by starting from $\nabla \log \bar \rho_0(x_0)$ and integrating over time (recall that $\bar \rho_0$ is known).
Here $\gJ_{f_{t}}$ denotes the Jacobian matrix of $f_t$.
Consequently, all we need to handle is the convolution term $K\ast \rho_t^f(\vx(t))$.

A common choice to approximate the convolution operation is via Monte-Carlo integration:
Let $\vy_{i}(t) \overset{\mathrm{iid}}{\sim} \rho_t^f$ for $i = 1, \ldots, N$ and denote an empirical approximation of $\rho_t^f$ by $\mu_N^{\rho_t^f} = \frac{1}{N}\sum_{i=1}^{N} \delta_{\vy_{i}(t)}$, where $\delta_{\vy_{i}(t)}$ denotes the Dirac measure at $\vy_{i}(t)$. 
We approximate the convolution term in equation (\ref{eqn_delta_x_t}) in different ways for the Coulomb and the  Biot-Savart interactions:
\vspace{-2mm}
\begin{enumerate}[wide, labelwidth=!, labelindent=0pt]
	\item For the Coulomb type kernel (\ref{eqn_coulomb_interaction}), we first approximate $K\ast \rho_t^f$ with $K_c\ast \rho_t^f$, where 
\begin{equation} \label{eqn_Coulomb_kernel_cutoff}
    K_c(\vx) \defi \begin{cases}
        K(\vx)     \quad &\text{if}\ \|\vx\|> c,\\
        0 \quad &\text{if}\ \|\vx\|\leq c.
    \end{cases}
\end{equation}
If $\rho_t^\vf$ is bounded in $\X$, we have
\begin{align*}
   \sup_{\vx\in\X} \|(K - K_c)\ast \rho_t^f(\vx)\| = \sup_{\vx\in\X} \|\int_{\|\vx -\vy\|\leq c} \frac{\vx - \vy}{\|\vx - \vy\|^{d}} \rho_t^f(\vy)\ud \vy \|
   \leq \|\rho_t^f\|_{\gL^\infty(\X)} \int_{\|\vy\|\leq c}\frac{1}{\|\vy\|^{d-1}}\ud \vy.
\end{align*}
To compute the integral on the right-hand side, we will switch to polar coordinates $(r, \psi)$:
\begin{equation} \label{eqn_polar_coordinate}
\int_{|\vy|\leq c}\frac{1}{|\vy|^{d-1}}\ud \vy = \int_0^c \ud r \frac{1}{r^{d-1}} \int_\Psi \ud\psi\ J_{(r, \psi)} \leq \int_0^c \ud r = c.
\end{equation}
Here, $J_{(r, \psi)}$ denotes the determinant of the Jacobian matrix resulting from the transformation from the Cartesian system to the polar coordinate system.
In inequality (\ref{eqn_polar_coordinate}), we utilize the fact that $J_{(r, \psi)} \leq r^{d-1}$, which allows us to cancel out the factor ${1}/{r^{d-1}}$.
Now that $K_c$ is bounded by $c^{-d+1}$, we can further approximate $K_c\ast \rho_t^\vf$ using $K_c\ast \mu_N^{\rho_t^\vf}$ with error of the order $O(c^{-d+1}/\sqrt{N})$. Altogether, we have $\sup_{\vx\in\X} \|K\ast \rho_t^f(\vx) - K_c\ast \mu_N^{\rho_t^\vf}(\vx)\| = O(c + c^{-d+1}/\sqrt{N})$ which can be made arbitrarily small for a sufficiently small $c$ and a sufficiently large $N$.
	\item For Biot-Savart interaction (2D Navier-Stokes equation), there are more structures to exploit and we can completely avoid the singularity: As noted by \cite{jabin2018quantitative}, the convolution kernel $K$ can be written in a divergence form:
	\begin{equation} \label{eqn_K_as_divergence}
		K = \nabla \cdot U, \text{ with } U(\vx) = \frac{1}{2\pi}\begin{bmatrix}
			-\arctan(\frac{\vx_1}{\vx_2}),& 0\\
			0,& \arctan(\frac{\vx_2}{\vx_1})
		\end{bmatrix},
	\end{equation}
	where the divergence of a matrix function is applied row-wise, i.e. $[K(\vx)]_i = \udiv\ U_i(\vx)$.
	Using integration by parts and noticing that the boundary integration vanishes on the torus, one has
	\begin{align*}
		K\ast \rho_t^f (\vx) =&\ \int K(\vy) \rho_t^f(\vx - \vy) d \vy = \int \nabla\cdot U(\vy) \rho_t^f(\vx - \vy) d \vy = \int U(\vy) \nabla \rho_t^f(\vx - \vy) d \vy \\
		=&\ \int U(\vx - \vy) \rho_t^f(\vy) \nabla \log \rho_t^f(\vy) d \vy = \E_{\vy\sim \rho_t^f(\vy)}[U(\vx - \vy) \nabla \log \rho_t^f(\vy)].
	\end{align*}
	If the score function $\nabla \log \rho_t^f$ is bounded, then the integrand in the expectation is also bounded. Therefore, we can avoid integrating singular functions and the Monte Carlo-type estimation $\frac{1}{N} \sum_{i=1}^N U(\vx - \vy_i(t)) \nabla \log \rho_t^f(\vy_i(t))$ is accurate for a sufficiently large value of N.
\end{enumerate}
With the above discussion, we can write $R(f_\theta; \vx_0)$ in an ODE-constrained form in a standard way, which due to space limitation is deferred to Appendix \ref{appendix_ODE_constrained_form}.
\begin{remark}
    Let $\ell_N(\theta)$ be the function we obtained using the above approximation of the convolution, where $N$ is the number of Monte-Carlo samples.
    The above discussion shows that $\ell_N(\theta)$ and $R(f_\theta; \vx_0)$ are close in the $\gL^\infty$ sense, which is hence sufficient when the \EINN\ loss is used as error quantification since only function value matters.
    When both $\ell_N(\theta)$ and $R(f_\theta; \vx_0)$ are in $C^2$, one can translate the closeness in function value to the closeness of their gradients.
    In our experiments, using $\nabla \ell_N(\theta)$ as an approximation of $\nabla_\theta R(f_\theta; \vx_0)$ gives very good empirical performance already. 
\end{remark}
\section{Analysis} \label{section_analysis}
In this section, we focus on the torus case, i.e. $\X = \Pi^d$ is a box with the periodic boundary condition.
This is a typical setting considered in the literature as the universal function approximation of NNs only holds over a compact set. Moreover, the boundary integral resulting from integration by parts vanishes in this setting, making it amenable for analysis purposes.
For completeness, we provide a discussion on the unbounded case, i.e. $\X = \sR^d$ in the Appendix \ref{appendix_unbounded}, which requires additional regularity assumptions.
Given the MVE (\ref{eqn_MVE}), if $K$ is bounded, it is sufficient to choose the Lyapunov functional $L(\rho_t^f, \bar \rho_t)$ as the KL divergence (please see Theorem \ref{theorem_bounded_K} in the appendix). But for the singular Coulomb kernel, we need also to consider the modulated energy as in \citep{serfaty2020mean}
\begin{equation}
    \text{(Modulated Energy)}\quad 
	\label{DefModEnergy}
	F(\rho, \bar \rho) \defi \frac 1 2 \int_{\mathcal{X}^2} g(x- y) \ud (\rho - \bar \rho)(x) \ud (\rho - \bar \rho )(y), 
\end{equation}
where $g$ is the fundamental solution to the Laplacian equation in $\mathbb{R}^d$, i.e. $- \Delta g = \delta_0$, and the Coulomb interaction reads $K= -\nabla g$ (see its closed form expression in equation (\ref{eqn_coulomb_interaction})). If we are only interested in the deterministic  dynamics with Coulomb interactions, i.e. $\nu =0$ in equation (\ref{eqn_MVE}), it suffices to choose $L(\rho_t^f, \bar \rho) $ as  $F(\rho_t^f, \bar \rho_t)$ (please see Theorem \ref{ThmCoul}). But if we consider the system with  Coulomb interactions and  diffusions, i.e. $\nu >0$, we shall combine the KL divergence and the modulated energy to form the modulated free energy as in \cite{bresch2019modulated}, which reads 
\begin{equation}
    \text{(Modulated Free Energy)}\quad 
	\label{DefModFree}
	E(\rho, \bar \rho) \defi \nu \KL (\rho, \bar \rho) +  F(\rho, \bar \rho). 
\end{equation}
This definition agrees with the physical meaning that ``Free Energy = Temperature $\times $ Entropy + Energy", and we note that the temperature is proportional to the diffusion coefficient $\nu$. We remark also for two probability densities $\rho$ and $\bar \rho$, $F(\rho, \bar \rho) \geq 0$ since by looking in the Fourier domain $F(\rho, \bar \rho) = \int \hat g(\xi) |\widehat{\rho - \bar \rho}(\xi)|^2  \ud \xi \geq 0$ as $\hat g(\xi) \geq 0$. Moreover, $F(\rho, \bar \rho)$ can be regarded as a negative Sobolev norm for $\rho- \bar \rho$, which metricizes weak convergence.\\
To obtain our main stability estimate, we first obtain the time evolution of the KL divergence.  
\begin{lemma}[Time Evolution of the KL divergence] \label{TimeEvolKLMV} 
	Given the hypothesis velocity field $f=f(t, x) \in C^1_{t, x}$. Assume that $(\rho_t^f)_{t \in [0, T]}$ and $(\bar \rho_t)_{t \in [0, T]}$ are classical solutions to equation (\ref{eqn_CE}) and equation (\ref{eqn_MVE_CE}) respectively.  It holds that (recall the definition of $\delta_t$ in equation (\ref{eqn_perturbation})) 
	\[
	\frac{\ud }{\ud t} \int_{\mathcal{X}} \rho^f_t \log \frac{\rho^f_t }{\bar \rho_t} = - \nu \int_{ \X}\rho^f_t |\nabla \log \frac{\rho^f_t}{\bar \rho_t}|^2 +  \int_{\X} \rho^f_t K * (\rho^f_t - \bar \rho_t ) \cdot \nabla  \log  \frac{\rho^f_t }{\bar \rho_t} +  \int_{\X} \rho^f_t \delta_t \cdot \nabla \log \frac{\rho^f_t }{\bar \rho_t }, 
	\]
	where $\X$ is the tours $\Pi^d$. All the integrands are evaluated at $\vx$. 
\end{lemma}

We refer the proof of this lemma and all other lemmas and theorems in this section to the appendix \ref{detailed proof}. We remark that to have the existence of classical solution $(\bar \rho_t)_{t \in [0, T]}$, we definitely need the regularity assumptions on $-\nabla V$ and on $K$. But the linear term $- \nabla V $ will not contribute to the evolution of the relative entropy. See \citep{jabin2018quantitative} for detailed discussions. \\
Similarly, we have the time evolution of the modulated energy as follows.
\begin{lemma}[Time evolution of the modulated energy] \label{ModuEnergyEvo} Under the same assumptions as in Lemma \ref{TimeEvolKLMV}, given the diffusion coefficient $\nu \geq 0$, it holds that (recall the definition of $\delta_t$ in equation (\ref{eqn_perturbation})) 
	\[
	\begin{split}
		\frac{\ud }{\ud t } F(\rho_t^f, \bar \rho_t)&  =  - \int_{\mathcal{X}} \rho_t^f \|K *(\rho_t^f - \bar \rho_t)\|^2 - \int_{\mathcal{X}} \rho_t^f \, \delta_t \cdot K * (\rho_t^f - \bar \rho_t ) + \nu \int_{\mathcal{X}} \rho^f_t \, K * (\rho_t^f - \bar \rho_t )\cdot \nabla \log \frac{\rho_t^f}{\bar \rho_t} \\
		&  - \frac{1}{2} \int_{\mathcal{X}^2} K(x-y) \cdot \Big( \mathcal{A}[\bar \rho_t](x) - \mathcal{A}[\bar \rho_t](y) \Big) \ud (\rho_t^f - \bar \rho_t )^{\otimes 2 }(x, y) \\
	\end{split}
	\]
	where we recall that the operator $\gA$ is defined in equation (\ref{eqn_operator_A}).
\end{lemma}

By  Lemma \ref{TimeEvolKLMV} and careful analysis, in particular by rewriting the Biot-Savart law in the divergence of a bounded matrix-valued function (\ref{eqn_K_as_divergence}), we obtain the following estimate for the 2D NSE.

\begin{theorem}[Stability estimate of the 2D NSE] \label{NSMainEstimate}	
    Notice that when $K$ is the Biot-Savart kernel, $\udiv K =0$. Assume that the initial data $\bar \rho_0 \in C^3(\Pi^d)$ and there exists $c>1$ such that $\frac 1 c \leq \bar \rho_0 \leq c$.  Assume further the hypothesis velocity field $f(t, x) \in C^1_{t, x}$. Then it holds that 
	\[
	\sup_{t \in [0, T]} \int_{\Pi^d} \rho_t^f \log \frac{\rho_t^f}{\bar \rho_t} \ud x \leq \frac{e^C}{\nu}  R(f), 
	\]
	where $C = \int_0^\infty     M(t) \ud t < \infty$ with 
	$M(t) \defi \|\nabla \log \bar \rho_t\|_{L^\infty}^2/2\nu + 2\Big\| {\nabla^2 \bar \rho_t}/{\bar \rho_t} \Big\|_{L^\infty}$. 
	
\end{theorem} 
We remark that given $\bar \rho_0$ is smooth enough and fully supported on $\X$, one can propagate the regularity to finally show the finiteness of $C$.
See detailed computations as in \cite{guillin2021uniform}. 
We give the complete proof in the appendix \ref{detailed proof}. This theorem tells us that as long as $R(f)$ is small, the KL divergence between $\rho_t^f$ and $\bar \rho_t$ is small and the control is uniform in time $t \in [0, T]$ for any $T$. Moreover, we highlight that $C$ is independent of $T$, and our result on the NSE is significantly better than the average-in-time and exponential-in-$T$ results from \citep{deerror}.\\
To treat the MVE (\ref{eqn_MVE}) with Coulomb interactions, we exploit the time evolution of the modulated free energy $E(\rho_t^f, \bar \rho_t)$. Indeed, combining Lemma \ref{TimeEvolKLMV} and Lemma \ref{ModuEnergyEvo}, we arrive at the following identity.  
\begin{lemma}[Time evolution of the modulated free energy]\label{TimeEvoMFE} 
	Under the same assumptions as in Lemma \ref{TimeEvolKLMV}, one has (recall the definitions of $\delta_t$ and $\gA$ in (\ref{eqn_perturbation}) and (\ref{eqn_operator_A}) respectively) 
	\[
	\begin{split}
		\frac{\ud }{\ud t } E(\rho_t^f, \bar \rho_t)  & = -\int_{\mathcal{X}} \rho_t^f  \Big|K * (\rho_t^f - \bar \rho_t) - \nu \nabla \log \frac{\rho_t^f }{\bar \rho_t}\Big|^2 - \int_{\mathcal{X}} \rho_t^f \, \delta_t \cdot \Big( K * (\rho_t^f - \bar \rho_t ) - \nu \nabla \log \frac{\rho_t^f}{\bar \rho_t }\Big) \\
		&  - \frac{1}{2} \int_{\mathcal{X}^2} K(x-y) \cdot \Big( \mathcal{A}[\bar \rho_t](x) - \mathcal{A}[\bar \rho_t](y) \Big) \ud (\rho_t^f - \bar \rho_t )^{\otimes 2 }(x, y).
	\end{split}
	\]
\end{lemma}
\vspace{-1mm}
Inspired by the mean-field convergence results as in \cite{serfaty2020mean} and \cite{bresch2019modulated}, we finally can control the growth of $E(\rho_t^f, \bar \rho_t)$ in the case when $\nu >0$,  and $F(\rho_t^f, \bar \rho_t)$ in the case when $\nu =0$. Note also that $E(\rho_t^f, \bar \rho_t )$ can also control the KL divergence when $\nu >0$. 
\begin{theorem} [Stability estimate of MVE with Coulomb interactions] \label{ThmCoul}
	Assume that for $t \in [0, T]$, the underlying velocity field $\mathcal{A}[\bar \rho_t](x)$ is Lipschitz  in $x$ and
	$\sup_{t \in [0, T]} \|\nabla \mathcal{A}[\bar \rho_t](\cdot)\|_{L^\infty} = C_1  < \infty.$ 
	Then there exists $C>0$ such that 
	\[
	\sup_{t \in [0, T]} \nu\,  \KL (\rho_t^f, \bar \rho_t) \leq \sup_{t \in [0, T]} E(\rho_t^f, \bar \rho_t) \leq \exp(C C_1 T) R(f).  
	\]
	In the deterministic case when $\nu =0$, under the same assumptions, it holds that 
	\[
	\sup_{t \in [0, T]} F( \rho_t^f, \bar \rho_t) \leq \exp(CC_1T ) R(f). 
	\]
\end{theorem}
\vspace{-1mm}
Recall the definition of the operator $\gA$ in \eqref{eqn_operator_A}. Given that $\mathcal{X}= \Pi^d$, and $\bar \rho_0$ is smooth enough and bounded from below, one can propagate regularity to obtain the Lipschitz condition for $\mathcal{A}[\bar \rho_t]$. See the proof and the discussion on the Lipschitz assumptions on $\mathcal{A}[\bar \rho_t](\cdot)$ in the appendix \ref{detailed proof}. 
\paragraph{Approximation Error of Neural Network}
Theorems \ref{NSMainEstimate} and \ref{ThmCoul} provide the error estimation guarantee for the proposed \EINN\ loss (\ref{eqn_self_consistency_potential}).
Suppose that we parameterize the velocity field $f=f_\theta$ with an NN parameterized by $\theta$, as we did in Section \ref{section_NN_parameterization} and let $\tilde f$ be the output of an optimization procedure when $R(f_\theta)$ is used as objective.
In order the explicitly quantify the mismatch between $\rho^{\tilde f}_t$ and $\bar \rho_t$, we need to quantify two errors: (i) Approximation error, reflecting how well the ground truth solution can be approximated among the NN function class of choice; (ii) Optimization error, involving minimization of a highly nonlinear non-convex objective. 
In the following, we show that for a function class $\gF$ with sufficient capacity, there exists at least one element $\hat f\in\gF$ that can reduce the loss function $R(\hat f)$ as much as desired.
We will not discuss how to identify such an element in the function class $\gF$ as it is independent of our research and remains possibly the largest open problem in modern AI research.
To establish our result, we make the following assumptions.
\begin{assumption} \label{ass_appendix_initial}
	$\rho_0$ is sufficiently regular, such that $\nabla \log \rho_0 \in \gL^\infty(\X)$ and $\bar f_t  = \mathcal{A}[\bar \rho_t] \in W^{2,\infty}(\X)$. $\nabla V$ is Lipschitz continuous. Here $W^{2,\infty}(\X)$ stands for the Sobolev norm of order $(2, \infty)$ over $\X$.
\end{assumption}
\vspace{-1mm}
We here again need to propagate the regularity for $f_t$ at least for a time interval $[0, T]$. It is easy to do so for the torus case, but for the unbounded domain, there are some technical issues to be  overcome. Similar assumptions are also needed in some mathematical works for instance in \cite{jabin2018quantitative}. 
We also make the following assumption on the capacity of the function class $\gF$, which is satisfied for example by NNs with tanh activation function \citep{DERYCK2021732}.
\begin{assumption} \label{ass_appendix_approximation}
	The function class is sufficiently large, such that there exists $\hat f \in \gF$ satisfying $\hat f_t \in \gC^3(\X)$ and $\|\hat f_t - \bar f_t\|_{W^{2, \infty}(\X)} \leq \epsilon$ for all $t\in[0, T]$.
\end{assumption}
\begin{theorem} \label{thm_approximation_error_NN}
	Consider the case where the domain is the torus. 
	Suppose that Assumptions \ref{ass_appendix_initial} and \ref{ass_appendix_approximation} hold. 
    For both the Coulomb and the Biot-Savart cases, there exists $\hat f\in\gF$ such that $R(\hat f) \leq C(T)\cdot(\epsilon \cdot\ln 1/\epsilon)^2$, where $C(T)$ is some constant independent of $\epsilon$. Here $R$ is the \EINN\ loss (\ref{eqn_self_consistency_potential}).
\end{theorem}
The major difficulty to overcome is the lack of Lipschitz continuity due to the singular interaction. We successfully address this challenge by establishing that the contribution of the singular region $(\|\vx\|\leq\epsilon)$ to $R(\hat f)$ can be bounded by $O((\epsilon \log \frac{1}{\epsilon})^2)$.
Please see the detailed proof in Appendix \ref{appendix_approximation_error_NN}.





\vspace{-1mm}
\section{Related Works on NN-based PDE solvers} \label{section_related_work}
Solving PDEs is a key aspect of scientific research, with a wealth of literature {\citep{evans2022partial}}.
Due to space limitations, a detailed discussion about the classical PDE solvers is deferred to Appendix \ref{appendix_related_work}.
In this section, we focus on the NN-based approaches as they are more related to our research.

As previously mentioned, PINN is possibly the most well-known method of this type.
PINN regards the solution to a PDE system as the root of the corresponding operators $\{\gD_i(\vg)\}_{i=1}^n$, and expresses the time and space boundary conditions as $\gB(\vg) = 0$, where $\vg$ is a candidate solution and $\gD_i$ and $\gB$ are operators acting on $\vg$.
Parameterizing $\vg= \vg_\theta$ using an NN, PINN optimizes its parameters $\theta$ by minimizing the residual $L(\theta) \defi \sum_{i=1}^n\lambda_i\|\gD_i(\vg_\theta)\|_{\gL^2(\X)}^2 + \lambda_0\|\gB(\vg_\theta)\|_{\gL^2(\X)}^2$. The hyperparameters $\lambda_i$ balance the validity of PDEs and boundary conditions under consideration and must be adjusted for optimal performance.
In contrast, \EINN\ requires no hyperparameter tuning.
PINN is versatile and can be applied to a wide range of PDEs, but its performance may not be as good as other NN-based solvers tailored for a particular class of PDEs, as it does not take into account other in-depth properties of the system, a phenomenon observed in the literature \citep{krishnapriyan2021characterizing,wang2}.
\citep{CiCP-28-2042} initiates the work of theoretically establishing the consistency of PINN by considering the {linear} elliptic and parabolic PDEs, for which they prove that a vanishing PINN loss $L(\theta)$ asymptotically implies $\vg_\theta$ recovers the true solution. A similar result is extended to the linear advection equations in \citep{shin2020error}.
Leveraging the stability of the operators $\gD_i$ (corresponding to PDEs of interest), non-asymptotic error estimations are established for linear Kolmogorov equations in \citep{de2022error}, for semi-linear and quasi-linear parabolic equations and the incompressible Euler in \citep{10.1093/imanum/drab093}, and for the NSE in \citep{deerror}. 
We highlight these non-asymptotic results are all average-in-time, meaning that even when the PINN loss is small the deviation of the candidate solution to the true solution may be significant at a particular timestamp $t \in [0, T]$. In comparison, our results are uniform-in-time, i.e. the supremum of the deviation is strictly bounded by the \EINN\ loss. Moreover, we show in Theorem \ref{NSMainEstimate}, for the NSE our error estimation holds for any $T$ uniformly, while the results in \citep{deerror} have an exponential dependence on $T$.

Recent work from \citet{zhang2022drvn} proposes the Random Deep Vortex Network (RDVN) method for solving the 2D NSE and achieves SOTA performance for this task. 
Let $\vu_t^\theta$ be an estimation of the interaction term $K\ast \rho_t$ in the SDE (\ref{eqn_MVE_particle}) and use $\rho_t^\theta$ to denote the law of the particle driven by the SDE $\ud \rmX_t = \vu_t^\theta(\rmX_t) \ud t  + \sqrt{2\nu}  \ud \rmB_t$. To train $\vu_t^\theta$, RDVN minimizes the loss $L(\theta) = \int_0^T\int_\X \|\vu_t^\theta(\vx) - K\ast \rho_t^\theta(\vx)\|^2_{\gL^2}\ud \vx \ud t$. Note that in order to simulate the SDE, one needs to discretize the time variable in loss function $L$. After training $\theta$, RDVN outputs $\rho_t^\theta$ as a solution. However, no error estimation guarantee is provided that controls the discrepancy between $\rho_t^\theta$ and $\rho_t$ using $L(\theta)$. 

\citet{shen22a} propose the concept of self-consistency for the FPE. However, unlike our work where the \EINN\ loss is derived via the stability analysis, they construct the potential $R(f)$ for the hypothesis velocity field $f$ by observing that the underlying velocity field $f^*$ is the fixed point of some velocity-consistent transformation $\gA$ and they construct $R(f)$ to be a more complicated Sobolev norm of the residual $f - \gA(f)$. In their result, they bound the Wasserstein distance between $\rho^f$ and $\rho$ by $R(f)$, which is weaker than our KL type control. The improved KL type control for the Fokker-Planck equation has also been discussed in \citep{boffi2023probability}. A very recent work \citep{li2023self} extends the self-consistency approach to compute the general Wasserstein gradient flow numerically, without providing further theoretical justification.

\vspace{-2mm}
\section{Experiments}
To show the efficacy and efficiency of the proposed approach, we conduct numerical studies on example problems that admit explicit solutions and compare the results with SOTA NN-based PDE solvers. 
The included baselines are PINN \citep{raissi2019physics} and DRVN \citep{zhang2022drvn}.
Note that these baselines only considered the 2D NSE.
We extend them to solve the MVE with the Coulomb interaction for comparison, and the details are discussed in Appendix \ref{appendix_implementation_of_baselines}.
\vspace{-.4cm}
\paragraph{Equations with an explicit solution} 
We consider the following two instances that admit explicit solutions. 
We verify these solutions in Appendix \ref{appendix_explicit_solution}.\\
\textit{Lamb-Oseen Vortex (2D NSE)} \citep{oseen1912uber}: Consider the whole domain case where $\gX = \sR^2$ and the Biot-Savart kernel (\ref{eqn_nse}). Let $\mathcal{N} ( \vmu , \mSigma)$ be the Gaussian distribution with mean $\vmu$ and covariance $\mSigma$.
If $\rho_0 = \mathcal{N}(0, \sqrt{2\nu t_0}\mI_2)$ for some $t_0 \geq 0$, then we have $\rho_t(\vx) = \mathcal{N}(0, \sqrt{2\nu (t+t_0)}\mI_2)$.\\
\textit{Barenblatt solutions (MVE)} \citep{serfaty2014mean}: Consider the 3D MVE with the Coulomb interaction kernel (\ref{eqn_coulomb_interaction}) with the diffusion coefficient set to zero, i.e. $d=3$ and $\nu=0$. Let $\Uniform[\sA]$ be the uniform distribution over a set $\sA$. Consider the whole domain case where $\gX = \sR^3$.
If $\rho_0 = \Uniform[\|\vx\|\leq (\frac{3}{4\pi}t_0)^{1/3}]$ for some $t_0 \geq 0$, then we have $\rho_t = \Uniform[\|\vx\|\leq (\frac{3}{4\pi}(t+t_0))^{1/3}]$.

\vspace{-.4cm}
\paragraph{Numerical results}
\begin{figure}
	\centering
	\begin{tabular}{c c c}
		\includegraphics[width=.27\columnwidth]{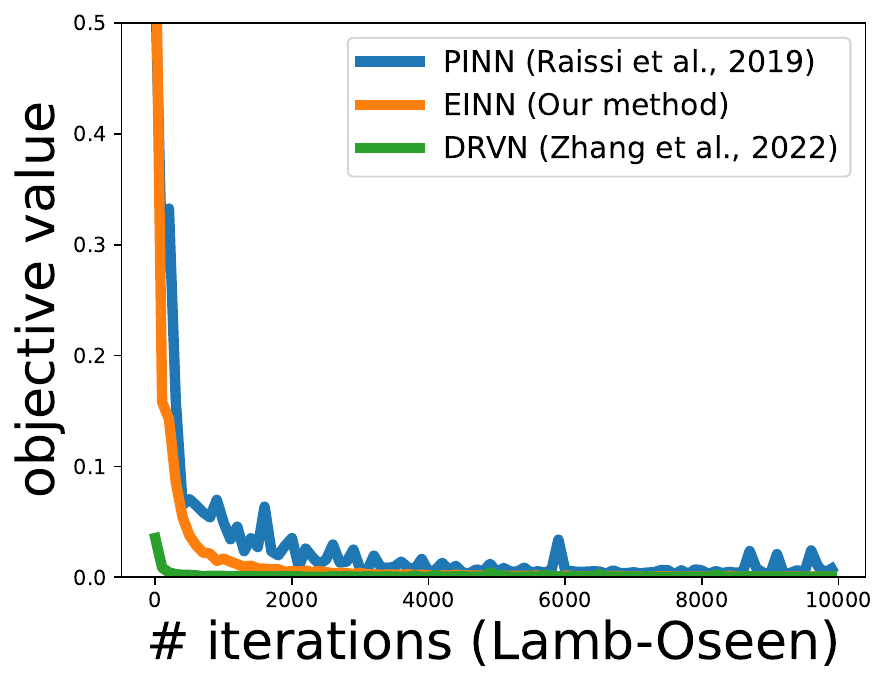} & \includegraphics[width=.27\columnwidth]{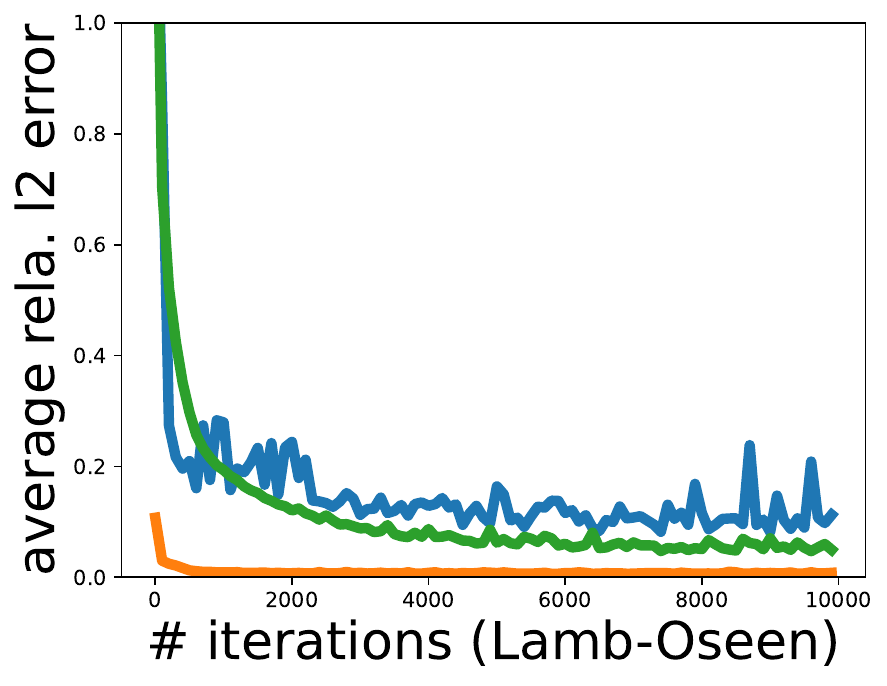} & \includegraphics[width=.27\columnwidth]{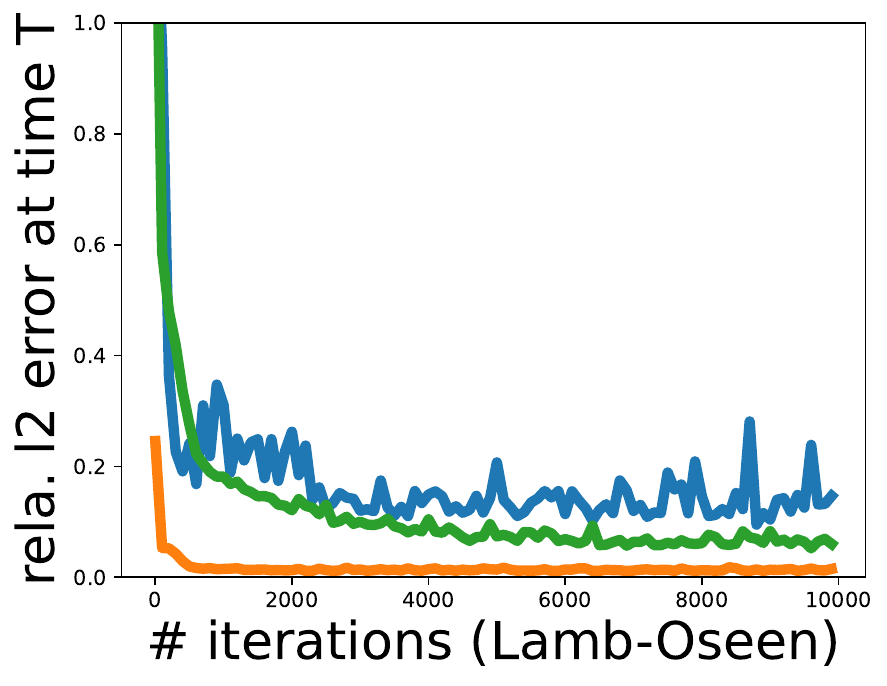} \\
		\includegraphics[width=.27\columnwidth]{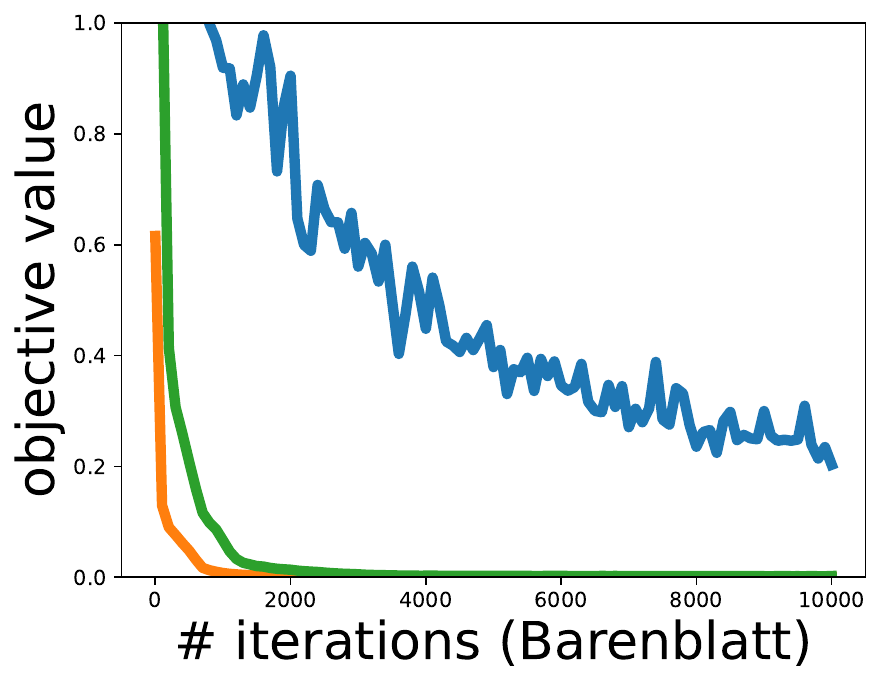} & \includegraphics[width=.27\columnwidth]{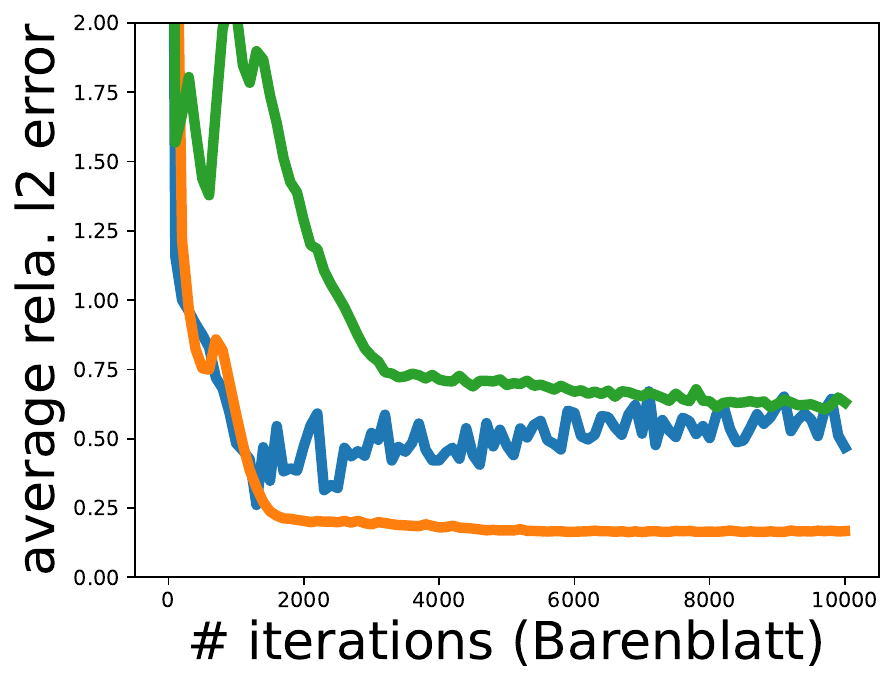} & \includegraphics[width=.27\columnwidth]{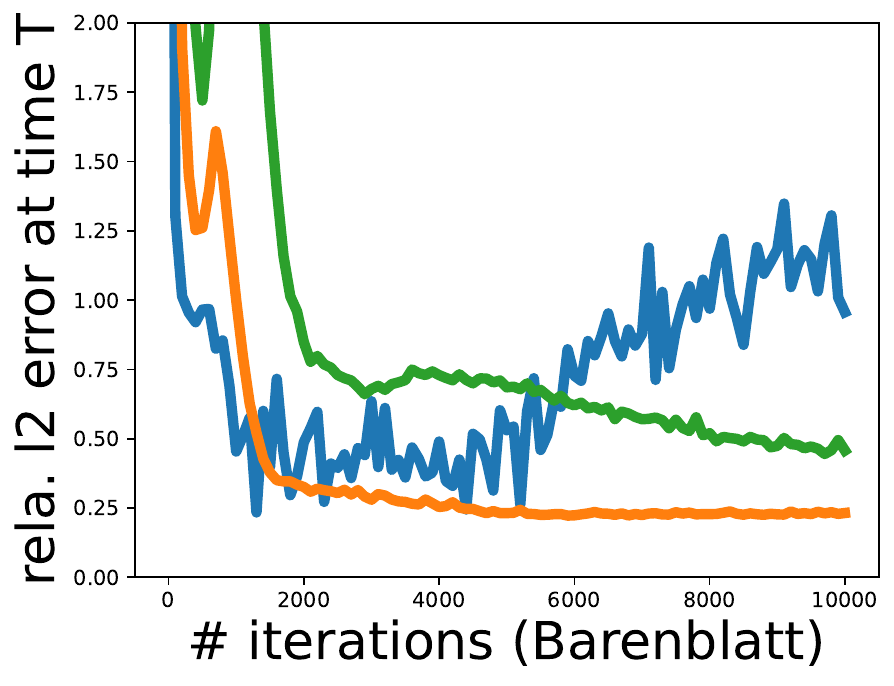}
	\end{tabular}
	\caption{The first row contains results for the 2D NSE and the second row contains the results for the 3D MVE with Coulomb interaction. The first column reports the objective losses, while the second and third columns report the average and last-time-stamp relative $\ell_2$ error. }
	\label{figure_experiment}
\end{figure}
We present the results of our experiments in Figure \ref{figure_experiment}, where the first row contains the result for the Lamb-Oseen vortex (2D NSE) and the second row contains the result for the Barenblatt model (3D MVE).
The explicit solutions of these models allow us to assess the quality of the outputs of the included methods. Specifically, given a hypothesis solution $\rho_t^f$, the ground truth $\bar \rho_t$ and the interaction kernel $K$, define the relative $\ell_2$ error at timestamp $t$ as 
$Q(t) \defi \int_\Omega {\|K\ast(\rho^f_t - \bar \rho_t) (\vx)\|}/{\|K\ast \bar \rho_t(\vx)\|} d \vx$,
where $\Omega$ is some domain where $\rho_t$ has non-zero density. We are particularly interested in the quality of the convolution term $K\ast\rho^f_t$ since it has physical meanings. In the Biot-Savart kernel case, it is the velocity of the fluid, while in the Coulomb case, it is the Coulomb field. We set $\Omega$ to be $[-2, 2]^2$ for the Lamb-Oseen vortex and to $[-0.1, 0.1]^3$ for the Barenblatt model. For both models, we take $\nu = 0.1$, $t_0 = 0.1$, and $T = 1$. The neural network that we use is an MLP with $7$ hidden layers, each of which has 20 neurons.\\
From the first column of Figure \ref{figure_experiment}, we see that the objective loss of all methods has substantially reduced over a training period of 10000 iterations. This excludes the possibility that a baseline has worse performance because the NN is not well-trained, and hence the quality of the solution now solely depends on the efficacy of the method.
From the second and third columns, we see that the proposed \EINN\ method significantly outperforms the other two methods in terms of the time-average relative $\ell_2$ error, i.e. $\frac{1}{T}\int_0^T Q(t) d t$ and the relative $\ell_2$ error at the last time stamp $Q(T)$. This shows the advantage of our method.
\vspace{-.3cm}

\paragraph{Conclusion}
By employing entropy dissipation of the MVE, we design a potential function for a hypothesis velocity field such that it controls the KL divergence between the corresponding hypothesis solution and the ground truth, for any time stamp within the period of interest.
Built on this potential, we proposed the \EINN\ method for MVEs with NN and derived the detailed computation method of the stochastic gradient, using the classic adjoint method. Through empirical studies on examples of the 2D NSE and the 3D MVE with Coulomb interactions, we show the significant advantage of the proposed method, when compared with two SOTA NN based MVE solvers. 

\begin{ack}
Zhenfu Wang is supported by the National Key R\&D Program of China, Project Number 2021YFA1002800, NSFC grant No.12171009, Young Elite Scientist Sponsorship Program by China Association for Science and Technology (CAST) No. YESS20200028 and the Fundamental Research Funds for the Central Universities (the start-up fund), Peking University.
Zebang Shen's work is supported by ETH research grant and Swiss National Science Foundation (SNSF) Project Funding No. 200021-207343.
\end{ack}
\bibliographystyle{abbrvnat}  
\bibliography{MVE}

\begin{thebibliography}{46}
\providecommand{\natexlab}[1]{#1}
\providecommand{\url}[1]{\texttt{#1}}
\expandafter\ifx\csname urlstyle\endcsname\relax
  \providecommand{\doi}[1]{doi: #1}\else
  \providecommand{\doi}{doi: \begingroup \urlstyle{rm}\Url}\fi

\bibitem[Boffi and Vanden-Eijnden(2023)]{boffi2023probability}
N.~M. Boffi and E.~Vanden-Eijnden.
\newblock Probability flow solution of the fokker--planck equation.
\newblock \emph{Machine Learning: Science and Technology}, 4\penalty0
  (3):\penalty0 035012, 2023.

\bibitem[Bresch et~al.(2019{\natexlab{a}})Bresch, Jabin, and
  Wang]{bresch2019mean}
D.~Bresch, P.-E. Jabin, and Z.~Wang.
\newblock On mean-field limits and quantitative estimates with a large class of
  singular kernels: application to the patlak--keller--segel model.
\newblock \emph{Comptes Rendus Mathematique}, 357\penalty0 (9):\penalty0
  708--720, 2019{\natexlab{a}}.

\bibitem[Bresch et~al.(2019{\natexlab{b}})Bresch, Jabin, and
  Wang]{bresch2019modulated}
D.~Bresch, P.-E. Jabin, and Z.~Wang.
\newblock Modulated free energy and mean field limit.
\newblock \emph{S{\'e}minaire Laurent Schwartz—EDP et applications}, pages
  1--22, 2019{\natexlab{b}}.

\bibitem[Cai et~al.(2021)Cai, Mao, Wang, Yin, and Karniadakis]{cai2021physics}
S.~Cai, Z.~Mao, Z.~Wang, M.~Yin, and G.~E. Karniadakis.
\newblock Physics-informed neural networks (pinns) for fluid mechanics: A
  review.
\newblock \emph{Acta Mechanica Sinica}, 37\penalty0 (12):\penalty0 1727--1738,
  2021.

\bibitem[Carrillo et~al.(2019)Carrillo, Craig, and
  Patacchini]{carrillo2019blob}
J.~A. Carrillo, K.~Craig, and F.~S. Patacchini.
\newblock A blob method for diffusion.
\newblock \emph{Calculus of Variations and Partial Differential Equations},
  58:\penalty0 1--53, 2019.

\bibitem[Carrillo et~al.(2022)Carrillo, Craig, Wang, and
  Wei]{carrillo2022primal}
J.~A. Carrillo, K.~Craig, L.~Wang, and C.~Wei.
\newblock Primal dual methods for wasserstein gradient flows.
\newblock \emph{Foundations of Computational Mathematics}, pages 1--55, 2022.

\bibitem[Cuomo et~al.(2022)Cuomo, Di~Cola, Giampaolo, Rozza, Raissi, and
  Piccialli]{cuomo2022scientific}
S.~Cuomo, V.~S. Di~Cola, F.~Giampaolo, G.~Rozza, M.~Raissi, and F.~Piccialli.
\newblock Scientific machine learning through physics--informed neural
  networks: where we are and what’s next.
\newblock \emph{Journal of Scientific Computing}, 92\penalty0 (3):\penalty0 88,
  2022.

\bibitem[De~Ryck and Mishra(2022)]{de2022error}
T.~De~Ryck and S.~Mishra.
\newblock Error analysis for physics-informed neural networks (pinns)
  approximating kolmogorov pdes.
\newblock \emph{Advances in Computational Mathematics}, 48\penalty0
  (6):\penalty0 1--40, 2022.

\bibitem[{De Ryck} et~al.(2021){De Ryck}, Lanthaler, and Mishra]{DERYCK2021732}
T.~{De Ryck}, S.~Lanthaler, and S.~Mishra.
\newblock On the approximation of functions by tanh neural networks.
\newblock \emph{Neural Networks}, 143:\penalty0 732--750, 2021.
\newblock ISSN 0893-6080.
\newblock \doi{https://doi.org/10.1016/j.neunet.2021.08.015}.
\newblock URL
  \url{https://www.sciencedirect.com/science/article/pii/S0893608021003208}.

\bibitem[De~Ryck et~al.(2023)De~Ryck, Jagtap, and Mishra]{deerror}
T.~De~Ryck, A.~Jagtap, and S.~Mishra.
\newblock Error estimates for physics-informed neural networks approximating
  the navier--stokes equations.
\newblock \emph{IMA Journal of Numerical Analysis}, 2023.

\bibitem[Erd{\H{o}}s and Yau(2017)]{erdHos2017dynamical}
L.~Erd{\H{o}}s and H.-T. Yau.
\newblock \emph{A dynamical approach to random matrix theory}, volume~28.
\newblock American Mathematical Soc., 2017.

\bibitem[Evans(2022)]{evans2022partial}
L.~C. Evans.
\newblock \emph{Partial differential equations}, volume~19.
\newblock American Mathematical Society, 2022.

\bibitem[Golse(2016)]{golse2016dynamics}
F.~Golse.
\newblock On the dynamics of large particle systems in the mean field limit.
\newblock \emph{Macroscopic and large scale phenomena: coarse graining, mean
  field limits and ergodicity}, pages 1--144, 2016.

\bibitem[Guillin et~al.(2021)Guillin, Bris, and
  Monmarch{\'e}]{guillin2021uniform}
A.~Guillin, P.~L. Bris, and P.~Monmarch{\'e}.
\newblock Uniform in time propagation of chaos for the 2d vortex model and
  other singular stochastic systems.
\newblock \emph{arXiv preprint arXiv:2108.08675}, 2021.

\bibitem[Gupta et~al.(2021)Gupta, Xiao, and Bogdan]{gupta2021multiwavelet}
G.~Gupta, X.~Xiao, and P.~Bogdan.
\newblock Multiwavelet-based operator learning for differential equations.
\newblock \emph{Advances in neural information processing systems},
  34:\penalty0 24048--24062, 2021.

\bibitem[Han et~al.(2018)Han, Jentzen, and E]{han2018solving}
J.~Han, A.~Jentzen, and W.~E.
\newblock Solving high-dimensional partial differential equations using deep
  learning.
\newblock \emph{Proceedings of the National Academy of Sciences}, 115\penalty0
  (34):\penalty0 8505--8510, 2018.

\bibitem[Jabin and Wang(2018)]{jabin2018quantitative}
P.-E. Jabin and Z.~Wang.
\newblock Quantitative estimates of propagation of chaos for stochastic systems
  with $\mathrm{W}^{-1, \infty}$ kernels.
\newblock \emph{Inventiones mathematicae}, 214\penalty0 (1):\penalty0 523--591,
  2018.

\bibitem[Johnson(2012)]{johnson2012numerical}
C.~Johnson.
\newblock \emph{Numerical solution of partial differential equations by the
  finite element method}.
\newblock Courier Corporation, 2012.

\bibitem[Jordan et~al.(1998)Jordan, Kinderlehrer, and
  Otto]{jordan1998variational}
R.~Jordan, D.~Kinderlehrer, and F.~Otto.
\newblock The variational formulation of the fokker--planck equation.
\newblock \emph{SIAM journal on mathematical analysis}, 29\penalty0
  (1):\penalty0 1--17, 1998.

\bibitem[Karniadakis et~al.(2021)Karniadakis, Kevrekidis, Lu, Perdikaris, Wang,
  and Yang]{karniadakis2021physics}
G.~E. Karniadakis, I.~G. Kevrekidis, L.~Lu, P.~Perdikaris, S.~Wang, and
  L.~Yang.
\newblock Physics-informed machine learning.
\newblock \emph{Nature Reviews Physics}, 3\penalty0 (6):\penalty0 422--440,
  2021.

\bibitem[Kovachki et~al.(2021)Kovachki, Li, Liu, Azizzadenesheli, Bhattacharya,
  Stuart, and Anandkumar]{kovachki2021neural}
N.~Kovachki, Z.~Li, B.~Liu, K.~Azizzadenesheli, K.~Bhattacharya, A.~Stuart, and
  A.~Anandkumar.
\newblock Neural operator: Learning maps between function spaces.
\newblock \emph{arXiv preprint arXiv:2108.08481}, 2021.

\bibitem[Krishnapriyan et~al.(2021)Krishnapriyan, Gholami, Zhe, Kirby, and
  Mahoney]{krishnapriyan2021characterizing}
A.~Krishnapriyan, A.~Gholami, S.~Zhe, R.~Kirby, and M.~W. Mahoney.
\newblock Characterizing possible failure modes in physics-informed neural
  networks.
\newblock \emph{Advances in Neural Information Processing Systems},
  34:\penalty0 26548--26560, 2021.

\bibitem[Li et~al.(2023)Li, Hurault, and Solomon]{li2023self}
L.~Li, S.~Hurault, and J.~Solomon.
\newblock Self-consistent velocity matching of probability flows.
\newblock \emph{arXiv preprint arXiv:2301.13737}, 2023.

\bibitem[Li et~al.(2020{\natexlab{a}})Li, Kovachki, Azizzadenesheli, Liu,
  Bhattacharya, Stuart, and Anandkumar]{li2020fourier}
Z.~Li, N.~Kovachki, K.~Azizzadenesheli, B.~Liu, K.~Bhattacharya, A.~Stuart, and
  A.~Anandkumar.
\newblock Fourier neural operator for parametric partial differential
  equations.
\newblock \emph{arXiv preprint arXiv:2010.08895}, 2020{\natexlab{a}}.

\bibitem[Li et~al.(2020{\natexlab{b}})Li, Kovachki, Azizzadenesheli, Liu,
  Bhattacharya, Stuart, and Anandkumar]{li2020neural}
Z.~Li, N.~Kovachki, K.~Azizzadenesheli, B.~Liu, K.~Bhattacharya, A.~Stuart, and
  A.~Anandkumar.
\newblock Neural operator: Graph kernel network for partial differential
  equations.
\newblock \emph{arXiv preprint arXiv:2003.03485}, 2020{\natexlab{b}}.

\bibitem[Long(1988)]{long1988convergence}
D.-G. Long.
\newblock Convergence of the random vortex method in two dimensions.
\newblock \emph{Journal of the American Mathematical Society}, 1\penalty0
  (4):\penalty0 779--804, 1988.

\bibitem[Majda et~al.(2002)Majda, Bertozzi, and Ogawa]{majda2002vorticity}
A.~J. Majda, A.~L. Bertozzi, and A.~Ogawa.
\newblock Vorticity and incompressible flow. cambridge texts in applied
  mathematics.
\newblock \emph{Appl. Mech. Rev.}, 55\penalty0 (4):\penalty0 B77--B78, 2002.

\bibitem[Mishra and Molinaro(2022)]{10.1093/imanum/drab093}
S.~Mishra and R.~Molinaro.
\newblock Estimates on the generalization error of physics-informed neural
  networks for approximating pdes.
\newblock \emph{IMA Journal of Numerical Analysis}, 43\penalty0 (1):\penalty0
  1--43, 01 2022.
\newblock ISSN 0272-4979.
\newblock \doi{10.1093/imanum/drab093}.
\newblock URL \url{https://doi.org/10.1093/imanum/drab093}.

\bibitem[Moukalled et~al.(2016)Moukalled, Mangani, Darwish, Moukalled, Mangani,
  and Darwish]{moukalled2016finite}
F.~Moukalled, L.~Mangani, M.~Darwish, F.~Moukalled, L.~Mangani, and M.~Darwish.
\newblock \emph{The finite volume method}.
\newblock Springer, 2016.

\bibitem[Oseen(1912)]{oseen1912uber}
C.~Oseen.
\newblock Uber die wirbelbewegung in einer reibenden flussigkeit.
\newblock \emph{Ark. Mat. Astro. Fys.}, 7, 1912.

\bibitem[Raissi et~al.(2019)Raissi, Perdikaris, and
  Karniadakis]{raissi2019physics}
M.~Raissi, P.~Perdikaris, and G.~E. Karniadakis.
\newblock Physics-informed neural networks: A deep learning framework for
  solving forward and inverse problems involving nonlinear partial differential
  equations.
\newblock \emph{Journal of Computational physics}, 378:\penalty0 686--707,
  2019.

\bibitem[Raissi et~al.(2020)Raissi, Yazdani, and Karniadakis]{raissi2020hidden}
M.~Raissi, A.~Yazdani, and G.~E. Karniadakis.
\newblock Hidden fluid mechanics: Learning velocity and pressure fields from
  flow visualizations.
\newblock \emph{Science}, 367\penalty0 (6481):\penalty0 1026--1030, 2020.

\bibitem[Rosenzweig(2022)]{rosenzweig2022mean}
M.~Rosenzweig.
\newblock Mean-field convergence of point vortices to the incompressible euler
  equation with vorticity in $\mathcal{L}^\infty$.
\newblock \emph{Archive for Rational Mechanics and Analysis}, 243\penalty0
  (3):\penalty0 1361--1431, 2022.

\bibitem[Serfaty(2020)]{serfaty2020mean}
S.~Serfaty.
\newblock Mean field limit for coulomb-type flows.
\newblock \emph{Duke Math. J.}, 169\penalty0 (15):\penalty0 2887--2935, 2020.

\bibitem[Serfaty and V{\'a}zquez(2014)]{serfaty2014mean}
S.~Serfaty and J.~L. V{\'a}zquez.
\newblock A mean field equation as limit of nonlinear diffusions with
  fractional laplacian operators.
\newblock \emph{Calculus of Variations and Partial Differential Equations},
  49:\penalty0 1091--1120, 2014.

\bibitem[Shen et~al.(2011)Shen, Tang, and Wang]{shen2011spectral}
J.~Shen, T.~Tang, and L.-L. Wang.
\newblock \emph{Spectral methods: algorithms, analysis and applications},
  volume~41.
\newblock Springer Science \& Business Media, 2011.

\bibitem[Shen et~al.(2022)Shen, Wang, Kale, Ribeiro, Karbasi, and
  Hassani]{shen22a}
Z.~Shen, Z.~Wang, S.~Kale, A.~Ribeiro, A.~Karbasi, and H.~Hassani.
\newblock Self-consistency of the fokker planck equation.
\newblock In P.-L. Loh and M.~Raginsky, editors, \emph{Proceedings of Thirty
  Fifth Conference on Learning Theory}, volume 178 of \emph{Proceedings of
  Machine Learning Research}, pages 817--841. PMLR, 02--05 Jul 2022.
\newblock URL \url{https://proceedings.mlr.press/v178/shen22a.html}.

\bibitem[Shin and Em~Karniadakis(2020)]{CiCP-28-2042}
J.~Shin, YeonjongDarbon and G.~Em~Karniadakis.
\newblock On the convergence of physics informed neural networks for linear
  second-order elliptic and parabolic type pdes.
\newblock \emph{Communications in Computational Physics}, 28\penalty0
  (5):\penalty0 2042--2074, 2020.

\bibitem[Shin et~al.(2020)Shin, Zhang, and Karniadakis]{shin2020error}
Y.~Shin, Z.~Zhang, and G.~E. Karniadakis.
\newblock Error estimates of residual minimization using neural networks for
  linear pdes.
\newblock \emph{arXiv preprint arXiv:2010.08019}, 2020.

\bibitem[Smith et~al.(1985)Smith, Smith, and Smith]{smith1985numerical}
G.~D. Smith, G.~D. Smith, and G.~D.~S. Smith.
\newblock \emph{Numerical solution of partial differential equations: finite
  difference methods}.
\newblock Oxford university press, 1985.

\bibitem[Villani et~al.(2009)]{villani2009optimal}
C.~Villani et~al.
\newblock \emph{Optimal transport: old and new}, volume 338.
\newblock Springer, 2009.

\bibitem[Wang et~al.(2022)Wang, Li, He, and Wang]{wang2}
C.~Wang, S.~Li, D.~He, and L.~Wang.
\newblock Is $\mathcal{L}^{2}$ physics informed loss always suitable for
  training physics informed neural network?
\newblock In \emph{Advances in Neural Information Processing Systems}, 2022.

\bibitem[Weinan et~al.(2021)Weinan, Han, and Jentzen]{weinan2021algorithms}
E.~Weinan, J.~Han, and A.~Jentzen.
\newblock Algorithms for solving high dimensional pdes: from nonlinear monte
  carlo to machine learning.
\newblock \emph{Nonlinearity}, 35\penalty0 (1):\penalty0 278, 2021.

\bibitem[Xiao et~al.(2023)Xiao, Cao, Yang, Gupta, Liu, Yin, Balan, and
  Bogdan]{xiao2023coupled}
X.~Xiao, D.~Cao, R.~Yang, G.~Gupta, G.~Liu, C.~Yin, R.~Balan, and P.~Bogdan.
\newblock Coupled multiwavelet neural operator learning for coupled partial
  differential equations.
\newblock \emph{arXiv preprint arXiv:2303.02304}, 2023.

\bibitem[Zhang et~al.(2018)Zhang, Han, Wang, and Car]{zhang2018deep}
L.~Zhang, J.~Han, H.~Wang, and R.~Car.
\newblock Deep potential molecular dynamics: a scalable model with the accuracy
  of quantum mechanics.
\newblock \emph{Physical review letters}, 120\penalty0 (14):\penalty0 143001,
  2018.

\bibitem[Zhang et~al.(2022)Zhang, Hu, Meng, Wang, Zhu, Chen, Ma, and
  Liu]{zhang2022drvn}
R.~Zhang, P.~Hu, Q.~Meng, Y.~Wang, R.~Zhu, B.~Chen, Z.-M. Ma, and T.-Y. Liu.
\newblock Drvn (deep random vortex network): A new physics-informed machine
  learning method for simulating and inferring incompressible fluid flows.
\newblock \emph{Physics of Fluids}, 34\penalty0 (10):\penalty0 107112, 2022.

\end{thebibliography}

\clearpage
\appendix
\section{Classical Methods for Solving MVEs} \label{appendix_related_work}
Solving partial differential equations (PDEs) is a key aspect of scientific research, with a wealth of literature in the field {\citep{evans2022partial}}. 
For the interest of this paper, we will only consider the methods that can be used to solve the MVE under consideration.

\paragraph{Categorize PDE solvers via solution representation.} To better understand the benefits of neural network (NN) based PDE solvers and to compare our approach with others, we categorize the literature based on the representation of the solution to the PDE. These representations can be roughly grouped into four categories:
\begin{itemize}[leftmargin=*]
    \item \textbf{1. Discretization-based representation:} The solution to the PDE is represented as discrete function values at grid points, finite-size cells, or finite-element meshes.
    \item \textbf{2. Representation as a combination of basis functions:} The solution to the PDE is approximated as a sum of basis functions, e.g. Fourier series, Legendre polynomials, or Chebyshev polynomials.
    \item \textbf{3. Representation using a collection of particles:} The solution to the PDE is represented as a collection of particles, each described by its weight, position, and other relevant information. 
    \item \textbf{4. NN-based representation:} NNs offer many strategies for representing the solution to the PDE, such as using the NN directly to represent the solution, using normalizing flow or GAN-based parameterization to ensure the non-negativity and conservation of mass of the solution or using the NN to parameterize the underlying dynamics of the PDE, such as the time-varying velocity field that drives the evolution of the system.
\end{itemize}

The drawback of the first three strategies is that a sparse representation\footnote{For example, sparser grid, cell or mesh with less granularity, fewer basis functions, fewer particles.} leads to reduced solution accuracy, while a dense representation results in increased computational and memory cost. 
NNs, as powerful function approximation tools, are expected to surpass these strategies by being able to handle higher-dimensional, less regular, and more complicated systems \citep{weinan2021algorithms}.

Given a representation strategy of the solution, an effective solver must exploit the underlying properties of the system to find the best candidate solution. Three-= notable properties that are utilized to design solvers are 
\begin{enumerate}[leftmargin=*,label=(\Alph*)]
    \item the PDE definition or weak formulation of the system,
    \item the SDE interpretation of the system,
    \item the variational interpretation, particularly the Wasserstein gradient flow interpretation.
\end{enumerate}

These properties are combined with the solution representations mentioned earlier to form different methods.
For example, the Finite Difference method \citep{smith1985numerical}, Finite Volume method \citep{moukalled2016finite}, and Finite Element method \citep{johnson2012numerical} represent the solution of partial differential equations (PDEs) by discretizing the solution and utilize the property (A), at least in their original form. On the other hand, recent work by \citet{carrillo2022primal} solves PDEs admitting a Wasserstein gradient flow structure using the classic JKO scheme \citep{jordan1998variational}, which is based on the property (C), and the solution is also represented via discretization. The Spectral method \citep{shen2011spectral} is a class of methods that exploits property (A) by representing the solution as a combination of basis functions.
The Random Vortex Method \citep{long1988convergence} is a highly successful method for solving the vorticity formulation of the 2D Navier-Stokes equation by exploiting property (C) and representing the solution with particles. The Blob method from \citet{carrillo2019blob} is another particle-based method for solving PDEs that describe diffusion processes, which also exploits property (C).

\section{Comparison with Neural Operator}
We thank the anonymous reviewers for pointing out the interesting research direction of neural operators \citep{xiao2023coupled,gupta2021multiwavelet,li2020neural,kovachki2021neural,li2020fourier}.
However, to highlight the major difference between EINN and the approach of the Neural Operator, it's worth noting that they consider completely different problem settings: EINN requires \emph{no pre-existing data} and the goal is to obtain the solution to a PDE by solely exploiting the structure of the equation itself. In contrast, the neural operator approach is \emph{data-driven}, i.e. it relies on the existence of configuration-solution pairs. Here, by configuration-solution pairs, we mean the correspondence between some configurations that determine the PDE, e.g. the initial condition or the viscosity parameter in the fluid dynamics problems, and the pre-existing solution to the PDE given the aforementioned configurations. Consequently, the neural operator approach is more like a regression problem where a neural network is trained to learn the abstract map between the configuration and the solution. In contrast, EINN is more like a numerical PDE solver.

Consequently, EINN and the approach of neural operator are two related but quite distinct research directions. They are related in the sense that EINN can provide the data (configuration-solution pairs) required by the neural operator approach. They are distinct since EINN requires no data a priori, while the neural operator approach is built on the supervised learning paradigm.
\section{More Details on the Experiments}
\subsection{Implementations of Baselines} \label{appendix_implementation_of_baselines}
\paragraph{Objectives of PINN}
\begin{itemize}[leftmargin=*]
	\item For the vorticity equation of the 2D Navier-Stokes equation, let $\vu: [0, T]\times \sR^2 \rightarrow \sR^2$ be the velocity field (this should not be confused with the velocity field of the continuity equation) such that $\nabla \cdot \vu = 0$, i.e. $\vu$ is divergence-free, and let $\omega = \nabla \times \vu \in \sR$ be the vorticity.
	We have
	\begin{align}
		\frac{\partial \omega}{\partial t} + \nabla \cdot \left(\omega\vu\right) =&\ \nu \Delta \omega,\\
		\omega =&\ \nabla \times \vu.
	\end{align}
	We use this form to construct the objective for the PINN method
	\begin{equation}
		\int_0^T \|\frac{\partial \omega}{\partial t} + \nabla \cdot \left(\omega\vu\right) - \nu \Delta \omega\|_{\gL(\Omega)^2}^2 + \|\omega - \nabla \times \vu\|_{\gL(\Omega)^2} d t,
	\end{equation}
	where $\gL^2(\Omega)$ denotes the functional $\gL^2$ norm on the domain $\Omega = [-2, 2]^2$.
	\item For the MVE with Coulomb interaction, let $g$ be the Coulomb potential defined in \eqref{eqn_coulomb_interaction}. We have that $\psi = g \ast \rho$ is the solution to the Poisson equation $\Delta \psi = - \rho$ and $K * \rho = - \nabla \psi$.
	We have
	\begin{align}
		\frac{\partial \rho}{\partial t} + \nabla \cdot \left(\rho\cdot (-\nabla \psi) \right) =&\ \nu \Delta \rho \\
		\Delta \psi =&\ -\rho.
	\end{align}
	Expand the the divergence to obtain
	\begin{align}
		\frac{\partial \rho}{\partial t} + \nabla \rho \cdot (-\nabla \psi) +  \rho\cdot (-\Delta \psi) =&\ \nu \Delta \rho \\
		\Delta \psi =&\ -\rho.
	\end{align}
	Now plug in the $\Delta \psi = -\rho$ to arrive at the following equivalent form
	\begin{align}
		\frac{\partial \rho}{\partial t} + \nabla \rho \cdot -\nabla \psi +  \rho^2 =&\ \nu \Delta \rho \\
		\Delta \psi =&\ -\rho.
	\end{align}
	We use this form to construct the objective for the PINN method.
	\begin{equation}
		\int_0^T \|\frac{\partial \rho}{\partial t} + \nabla \rho \cdot -\nabla \psi +  \rho^2 - \nu \Delta \rho\|_{\gL^2(\Omega)}^2 + \|\Delta \psi + \rho\|_{\gL^2}^2,
	\end{equation}
	where $\gL^2(\Omega)$ denotes the functional $\gL^2$ norm on the domain $\Omega = [-1, 1]^2$.
\end{itemize}

\paragraph{DRVN} In the original paper \citep{zhang2022drvn}, only the Biot-Savart kernel is concerned. We extend the DRVN method to the Coulomb case by setting $K$ to be the kernel defined in \eqref{eqn_coulomb_interaction}.

\subsection{Examples with an Explicit Solution} \label{appendix_explicit_solution}
In this section, we verify the explicit solutions discussed in the experiment section.
\paragraph{Lamb-Oseen Vortex on the whole domain $\sR^2$.}
Recall that we consider the 2D Navier-Stokes equation (the MVE with the Biot-Savart interaction kernel (\ref{eqn_nse})).
The Lamb-Oseen Vortex model states that, if $\rho_0 = \mathcal{N}(0, \sqrt{2\nu t_0}\mI_2)$ for some $t_0 \geq 0$, then we have $\rho_t(\vx) = \mathcal{N}(0, \sqrt{2\nu (t+t_0)}\mI_2)$.

To verify this, define $\vu_t(\vx) = \frac{1}{\sqrt{\nu (t+t_0)}}\vv(\frac{\vx}{\sqrt{\nu (t+t_0)}})$, where 
\begin{equation}
	\vv(\vx) = \frac{1}{2\pi}\frac{\vx^{\perp}}{\|\vx\|^2}\left(1 - \exp(-\frac{1}{4}\|\vx\|^2)\right).
\end{equation}
One can easily check that $\nabla \cdot \vu_t \equiv 0$ and hence there exists a function $\psi_t$ such that $\nabla^\perp \psi_t = -\vu_t$, where $\nabla^\perp$ denotes the perpendicular gradient, defined as $\nabla^\perp = (-\partial_{\vx_2}, \partial_{\vx_2})$, and $\psi_t$ is called the stream function in the literature of fluid dynamics.
Moreover, one can verify that $\nabla \times \vu_t = \rho_t$ where $\nabla \times$ denotes the curl of a 2D velocity field, defined as $\nabla \times \vu(\vx) = \partial \vu_2 / \partial \vx_1 - \partial \vu_1/\partial \vx_2$.
Together we have
\begin{equation}
	\Delta \psi_t = -\rho_t,
\end{equation}
i.e., the stream function $\psi_t$ is the solution to the 2D Poisson equation with a source term $\rho_t$.

Under the boundary condition that $\psi_t(\vx) \rightarrow 0$ for $\|\vx\|\rightarrow \infty$, we can express $\psi_t$ via the unique Green function $G(\vx) = \frac{1}{2\pi}\ln \|\vx\|$ as 
\begin{equation}
	\psi_t(\vx) = G \ast \rho_t = \frac{1}{2\pi}\int \ln\|\vx - \vy\| \rho_t(\vy) d \vy.
\end{equation}
Consequently, by taking the perpendicular gradient, we obtain
\begin{equation}
	\vu_t = \nabla^\perp \psi_t = \frac{1}{2\pi}\int \frac{(\vx - \vy)^\perp}{\|\vx - \vy\|^2} \rho_t(\vy) d \vy = K \ast \rho_t.
\end{equation}
Finally, by plugging the expressions of $\rho_t$ and $\vu_t = K \ast \rho_t$ in the MVE (\ref{eqn_MVE}), we verified the Lamb-Oseen vortex.
%

\paragraph{Barenblatt solutions for the MVE with Coulomb Interaction.} 
Recall that we consider the MVE with the Coulomb interaction kernel (\ref{eqn_coulomb_interaction}) for $d=3$ and set the diffusion coefficient $\nu = 0$, i.e.
\begin{equation}
	\frac{\partial \rho_t}{\partial t} + \nabla \cdot \left( \rho_t \cdot -\nabla \psi_t \right) = 0
\end{equation}
where $\psi_t$ is the solution to the Poisson equation $\Delta \psi_t = -\rho_t$. The Barenblatt solution of the above MVE is stated as follows: If $\rho_0 = \Uniform[\|\vx\|\leq (\frac{3}{4\pi}t_0)^{1/3}]$ for some $t_0 \geq 0$, then we have 
\begin{equation}
	\rho_t = \Uniform[\|\vx\|\leq (\frac{3}{4\pi}(t+t_0))^{1/3}]
\end{equation}
We now verify this solution.

Recall that the volume of a three dimensional Euclidean ball with radius $R$ is $\frac{4\pi}{3}R^3$. Hence we can write the density function as $\rho_t(x) = \frac{1}{t+t_0} \chi_{\|\vx\|\leq (\frac{3}{4\pi}(t+t_0))^{1/3}}$, where $\chi_\sX$ is a function that takes value $1$ for $\vx \in \sX$ and takes value $0$ for $\vx \notin \sX$.
Take 
\begin{equation}
	\psi_t(\vx) = \begin{cases}
		\frac{2(\frac{3}{4\pi}(t+t_0))^{2/3}-\|\vx\|^2}{6(t+t_0)}, & \|\vx\| \leq (\frac{3}{4\pi}(t+t_0))^{1/3},\\
		\frac{1}{8\pi\|\vx\|}, & \|\vx\| > (\frac{3}{4\pi}(t+t_0))^{1/3}.
	\end{cases}
\end{equation}
It can be verified that the Poisson equation $\Delta \psi_t = -\rho_t$ holds (note that $\Delta \|\vx\|^{-1} = 0$, i.e. $\|\vx\|^{-1}$ is a harmonic function for $d=3$).
Consequently, for a fixed time stamp $t$ and any $\|\vx\| \leq (\frac{3}{4\pi}(t+t_0))^{1/3}$ we have 
\begin{align}
	\frac{\partial \rho_t}{\partial t}(\vx) + \nabla \cdot \left( \rho_t(\vx) \cdot -\nabla \psi_t(\vx) \right) = - \frac{1}{(t+t_0)^2} + \frac{1}{(t+t_0)^2} = 0,
\end{align}
which verifies this solution.

\clearpage
\section{Adjoint Method} \label{appendix_adjoint_method}
Consider the ODE system
\begin{align*}
	\dot s(t) =&\ \psi(s(t), t, \theta) \\
	s(0) =&\ s_0,
\end{align*}
and the objective loss
\begin{equation}
	\ell(\theta) = \int_0^T g(s(t), t, \theta) \ud t.
\end{equation}
The following proposition computes the gradient of $\ell$ w.r.t. $\theta$.
We omit the parameters of the functions for succinctness. We note that all the functions in the integrands should be evaluated at the corresponding time stamp $t$, e.g. $b^\top \frac{\partial h}{\partial \theta}\ud t$ abbreviates for $b(t)^\top \frac{\partial}{\partial \theta}h(\xi(t), x(t), t, \theta)\ud t$.
\begin{proposition}
	\begin{equation}
		\frac{\ud \ell}{\ud \theta} = \int_{0}^T a^\top \frac{\partial\psi}{\partial\theta} + \frac{\partial g}{\partial\theta}\ud t.
	\end{equation}
	where $a(t)$ is solution to the following final value problems
	\begin{equation}
		\dot a^\top + a^\top \frac{\partial\psi}{\partial s} + \frac{\partial g}{\partial s} = 0, a(T) = 0, 
	\end{equation}
\end{proposition}
\begin{proof}
	Let us define the Lagrange multiplier function (or the adjoint state) $a(t)$ dual to $s(t)$.
	Moreover, let $L$ be an augmented loss function of the form
	\begin{equation}
		L = \ell - \int_0^T a^\top(\dot s - \psi) \ud t.
	\end{equation}
	Since we have $\dot s(t) = \psi(s(t), t, \theta)$ by construction, the integral term in $L$ is always null and $a$ can be freely assigned while maintaining $\ud L/\ud \theta = \ud \ell/\ud \theta$.
	Using integral by part, we have
	\begin{equation}
		\int_0^T a^\top\dot s\ \ud t = a(t)^\top s(t)\vert_0^T - \int_0^T s^\top \dot a\ \ud t.
	\end{equation}
	We obtain
	\begin{align}
		L = - a(t)^\top s(t)\vert_0^T + \int_0^T \dot a^\top s + a^\top \psi + g\ \ud t.
	\end{align}
	
	Now we compute the gradient of $L$ w.r.t. $\theta$ as
	\begin{equation*}
		\frac{\ud \ell}{\ud \theta} =  \frac{\ud L}{\ud \theta} = - a(T)^\top\frac{\ud x(T)}{\ud \theta}  + \int_0^T \dot a^\top \frac{\ud s}{\ud \theta} + a^\top \left(\frac{\partial\psi}{\partial\theta} + \frac{\partial\psi}{\partial s} \frac{\ud s}{\ud \theta} \right) \ud t
		+ \int_0^T \frac{\partial g}{\partial s} \frac{\ud s}{\ud \theta} +  \frac{\partial g}{\partial\theta}\ud t,
	\end{equation*}
	which by rearranging terms yields to
	\begin{align*}
		\frac{\ud \ell}{\ud \theta} = \frac{\ud L}{\ud \theta} = - a(T)^\top\frac{\ud x(T)}{\ud \theta} + \int_{0}^T a^\top \frac{\partial \psi}{\partial \theta} +  \frac{\partial g}{\partial \theta}\ud t
		+ \int_0^T \left(\dot a^\top + a^\top \frac{\partial \psi}{\partial s} +  \frac{\partial g}{\partial s}\right)\frac{\ud s}{\ud \theta} \ud t.
	\end{align*}
	Now by taking $a$ satisfying the \emph{final} value problems
	\begin{equation}
		\dot a^\top + a^\top \frac{\partial \psi}{\partial s} + \frac{\partial g}{\partial s} = 0, a(T) = 0, 
	\end{equation}
	we derive the result
	\begin{equation}
		\frac{\ud \ell}{\ud \theta} = \int_{0}^T a^\top \frac{\partial \psi}{\partial \theta} + \frac{\partial g}{\partial \theta}\ud t.
	\end{equation}
\subsection{Writing the Trajectory-wise Loss (\ref{eqn_trajectory_wise_loss}) in an ODE-constrained form} \label{appendix_ODE_constrained_form}
We are now ready to write $R(f_\theta; \vx_0)$ in an ODE-constrained form. Define the state $\vs(t)$, the initial condition $\vs_0$ and the transition function $\psi$ as follows: Let
\begin{equation}
	\vs(t) = \left[\vx(t), \xi(t), \{\vy_i(t)\}_{i=1}^N, \{\zeta_i(t)\}_{i=1}^N\right],
\end{equation}
with $\xi(t) = \nabla\log\rho_t^f(\vx(t))$ and $\zeta_i(t) = \nabla\log\rho_t^f(\vy_i(t))$.  Take  the initial condition 
\begin{equation}
	\vs_0 = \left[\vx_0, \xi_0, \{\vy_i(0)\}_{i=1}^N, \{\zeta_i(0)\}_{i=1}^N\right]
\end{equation}
with $\xi_0 = \nabla\log \bar \rho_0(\vx_0)$, $\zeta_i(0) = \nabla\log\bar\rho_0(\vy_i(0))$, and $\vy_i(0) \stackrel{iid}{\sim} \bar \rho_0$;
and define the function
\begin{equation}
	\psi(t, s(t); \theta) = [f_t(\vx(t); \theta),\ h_t(\vx(t), \xi(t); \theta),\ \{f_t(\vy_i(t); \theta)\}_{i=1}^N,\ \{h_t(\vy_i(t), \zeta_i(t); \theta)\}_{i=1}^N],
\end{equation}
where $h(\va, \vb; \theta) = - \nabla  \left(\udiv f_{t}(\va; \theta)\right) -   \gJ^\top_{f_{t}}(\va; \theta) \vb$ (derived from \eqref{eqn_dynamics_of_score}).
Finally, define
\begin{equation}
	g(t, \vs(t); \theta) = \|f(t, \vx(t); \theta) - \left(-\nabla V(\vx(t)) + E(t, \vs(t)) - \nu \xi(t)\right)\|^2,
\end{equation}
where the estimator $E(t, \vs)$  of the convolution term is defined as
\begin{equation}
	E(t, \vs(t)) = \begin{cases}
		\frac{1}{N} \sum_{i=1}^N K_c(\vx(t) -  \vy_i(t)) & \text{the Coulomb case}, \\
		\frac{1}{N} \sum_{i=1}^N U(\vx - \vy_i(t)) \zeta_i(t) & \text{the Biot-Savart case}.
	\end{cases}
\end{equation}
We recall the definition of $U$ in \eqref{eqn_K_as_divergence} and the definition of $K_c$ in \eqref{eqn_Coulomb_kernel_cutoff}.
\end{proof}
\clearpage
\section{Detailed Proofs}\label{detailed proof}

\begin{proof}[Proof of Lemma  \ref{TimeEvolKLMV}] 
	Recall the McKean-Vlasov  \eqref{eqn_MVE_CE} and the continuity  \eqref{eqn_CE_as_MVE}. We simply write that $\rho_t = \rho_t^f$ and omit the integration domain $\X$.  Then 
	\[
	\begin{split}
		& \frac{\ud }{\ud t } \int \rho_t \log \frac{\rho_t }{\bar \rho_t }  = \int \partial_t \rho_t \log \frac{\rho_t}{\bar \rho_t} +  \int \rho_t   \partial_t \log \rho_t - \int \rho_t \partial_t \log \bar \rho_t  \\
		& = - \int  \udiv \Big(    \rho_t  \Big(\Big[- \nabla V (x) + K *  \rho_t -  \nu \nabla \log  \rho_t \Big] + \delta_t  \Big) \Big) \log \frac{\rho_t}{\bar \rho_t} \\
		& +  \int \frac{\rho_t}{\bar \rho_t}  \udiv \Big(   \bar \rho_t \Big(- \nabla V (x) + K * \bar \rho_t - \nu \nabla \log \bar \rho_t \Big)\Big), 
	\end{split}
	\]
	where we note that $\int \rho_t \partial_t \log  \rho_t= \int \partial_t \rho_t = 0$ since the total mass is preserved over time. 
	By integration by parts, one has 
	\[
	\begin{split}
		& \frac{\ud }{\ud t } \int \rho_t \log \frac{\rho_t }{\bar \rho_t }  = I_1 + I_2 + I_3+ \int \rho_t \delta_t \cdot \nabla \log \frac{\rho_t}{\bar \rho_t},
	\end{split}
	\]
	where $I_1, I_2, I_3$ denote the linear, nonlinear interaction, and diffusion parts separately. More precisely, by integration by parts,
	\[
	\begin{split}
		I_1 & = \int \udiv (\rho_t \nabla V(x)) \log \frac{\rho_t}{\bar \rho_t} - \int \frac{\rho_t}{\bar \rho_t} \udiv (\bar \rho_t \nabla V(x))  \\
		& = - \int \rho_t \nabla V(x) \cdot \nabla \log \frac{\rho_t}{\bar \rho_t} + \int \bar \rho_t  \nabla  \frac{\rho_t }{\bar \rho_t} \cdot  \nabla V(x) = 0. 
	\end{split}
	\]
	And 
	\[
	\begin{split} 
		I_2 & = - \int \udiv (\rho_t K * \rho_t ) \log \frac{\rho_t}{\bar \rho_t } + \int \frac{\rho_t }{\bar \rho_t} \udiv(\bar \rho_t K * \bar \rho_t)  \\
		& =  \int \rho_t K * \rho_t \nabla \log \frac{\rho_t}{\bar \rho_t } - \int \bar \rho_t K * \bar \rho_t \cdot \nabla \frac{\rho_t}{\bar \rho_t } \\
		& = \int \rho_t \nabla \log \frac{\rho_t}{\bar \rho_t} \cdot K * (\rho_t - \bar \rho_t). 
	\end{split} 
	\]
	Given that the kernel $K$ is divergence free, that is $\udiv K = 0$, one further has 
	\begin{equation}\label{NSE_new}
	\begin{split}
		I_2 & = - \int \rho_t \nabla \log \bar \rho_t  \cdot K *(\rho_t - \bar \rho_t ) + \int \nabla \rho_t \cdot K *(\rho_t - \bar \rho_t ) \\
		& = - \int \rho_t \nabla \log \bar \rho_t \cdot K * (\rho_t - \bar \rho_t). 
	\end{split}
	\end{equation}
	Note that this modification will be used in the proof in the 2D Navier-Stokes case. 
	Finally, all diffusion terms sum up to $I_3$ which can be further simplified as 
	\[
	\begin{split} 
		I_3  & = \nu \int \udiv (\rho_t \nabla \log  \rho_t ) \log \frac{\rho_t }{\bar \rho_t} - \nu \int \frac{\rho_t }{\bar \rho_t} \udiv (\bar \rho_t \nabla \log \bar \rho_t ) \\
		& = - \nu \int \rho_t \nabla \log \rho_t \cdot \nabla \log \frac{\rho_t}{\bar \rho_t} + \nu \int \bar \rho_t \nabla \log \bar  \rho_t \cdot  \nabla \frac{\rho_t }{\bar \rho_t } \\
		& = - \nu \int \rho_t |\nabla \log \frac{\rho_t }{\bar \rho_t}|^2. 
	\end{split} 
	\]
	We thus complete the proof of Lemma \ref{TimeEvolKLMV}.

\end{proof}

\begin{proof}[Proof of Lemma \ref{ModuEnergyEvo}]Recall that $K= - \nabla g$.  For simplicity, we write that $\rho_t = \rho_t^f$. Then 
	\[
	\begin{split} 
		& \frac{\ud }{\ud t } F(\rho_t, \bar \rho_t) = \frac{\ud }{\ud t } \frac{1 }{2} \int_{\mathcal{X}^2} g (x- y) \ud (\rho_t - \bar \rho_t)^{\otimes 2 } (x, y)  \\
		& = \int_\mathcal{X} g *(\rho_t - \bar \rho_t) (x) \big(\partial_t \rho_t (x) - \partial_t \bar \rho_t (x)  \Big) \ud x  \\
		& = \int g * (\rho_t - \bar \rho_t ) (x) \udiv \Big\{  \rho_t \Big( [\nabla V (x) - K * \rho_t + \nu \nabla  \log \rho_t ]  -  \delta_t \Big)   \\ & \quad \qquad \qquad \qquad \qquad - \bar \rho_t \Big( \nabla V(x) - K * \bar \rho_t + \nu \log \bar \rho_t \Big) \Big\} \\
		& = J_1+ J_2+ J_3 + J_4, 
	\end{split} 
	\]
	where $J_1, J_2, J_3, J_4$ denote the perturbation term, the linear difference term, the nonlinear difference term, and the diffusion term respectively. The perturbation term $J_1$ reads
	\[
	J_1 = - \int_{\mathcal{X}} g*(\rho_t - \bar \rho_t ) \udiv (\rho_t \delta_t ) = - \int_{\mathcal{X}}  \rho_t  K * (\rho_t - \bar \rho_t )  \cdot \delta_t. 
	\]
	By integration by parts, the linear difference term can be written as 
	\[
	\begin{split}
		&J_2 = \int_{\mathcal{X}}  g *(\rho_t - \bar \rho_t ) \udiv\Big( (\rho_t - \bar \rho_t ) \nabla V \Big)  = \int_{\mathcal{X}} K * (\rho_t - \bar \rho_t) (\rho_t - \bar \rho_t ) \nabla V \\
		&  = \frac{1}{2} \int_{\mathcal{X}^2} K (x - y)(\nabla V (x) - \nabla  V(y)) \ud (\rho_t - \bar \rho_t )^{\otimes 2}(x, y), 
	\end{split}
	\]
	where the last equality is true since $K = - \nabla g $ 
	is an odd function and we do the symmetrization trick, i.e. exchanging the role of $x$ and $y$ to another term and then taking the average. 
	
	The nonlinear difference term reads 
	\[
	\begin{split}
		& J_3 = - \int_{\mathcal{X}} g*(\rho_t - \bar \rho_t) \udiv\Big(\rho_t K * \rho_t - \bar \rho_t K * \bar \rho_t  \Big)  \\
		& = - \int_{\mathcal{X}} K * (\rho_t - \bar \rho_t ) (\rho_t K * (\rho_t - \bar \rho_t) - \int_{\mathcal{X}}  K * (\rho_t - \bar \rho_t) (\rho_t - \bar \rho_t ) K * \bar \rho_t  \\
		& = - \int_{\mathcal{X}}  \rho_t |K * (\rho_t - \bar \rho_t)|^2 - \frac{1}{2} \int K(x-y) (K * \bar \rho_t (x) - K * \bar \rho_t (y) ) \ud (\rho_t - \bar \rho_t)^{\otimes 2}(x, y), 
	\end{split}
	\]
	where again in the last term we do the symmetrization. 
	
	The diffusion term reads 
	\[
	\begin{split}
		& J_4 =  \nu \int g * (\rho_t - \bar \rho_t ) \udiv \Big( \rho_t  \nabla \log \rho_t - \bar \rho_t \nabla \log \bar \rho_t  \Big)  \\
		& = \nu \int K * (\rho_t - \bar \rho_t)  \rho_t \nabla \log \frac{\rho_t }{\bar \rho_t }   + \nu \int K * (\rho_t - \bar \rho_t ) (\rho_t - \bar \rho_t ) \nabla \log \bar \rho_t  \\
		& = \nu \int_{\mathcal{X}} \rho_t K *(\rho_t - \bar \rho_t)  \cdot \nabla \log \frac{\rho_t}{\bar \rho_t}  \\
		& \qquad + \frac{\nu }{2} \int_{\mathcal{X}^2} K (x-y)(\nabla \log \bar \rho_t (x) - \nabla \log \bar \rho_t (y)) \ud (\rho_t - \bar \rho_t )^{\otimes 2}. 
	\end{split} 
	\]
	To sum it up, we prove the thesis.

\end{proof}

\subsection{Proof of the 2D Navier-Stokes case} 
Now we proceed to control the growth of the KL divergence  $\mathbf{KL}(\rho_t^f \vert \bar \rho_t )$ for the 2D Navier-Stokes case. 
Since the Biot-Savart law is divergence free, by \eqref{NSE_new} in the proof of Lemma \ref{TimeEvolKLMV}, one has 
\begin{equation}\label{KL_NS_Evo}
	\frac{\ud }{\ud t} \int_{\Pi^d} \rho_t \log \frac{\rho_t }{\bar \rho_t} = - \nu \int_{ \Pi^d}\rho_t |\nabla \log \frac{\rho_t}{\bar \rho_t}|^2 - \int_{\Pi^d} \rho_t K * (\rho_t - \bar \rho_t ) \cdot \nabla  \log \bar   \rho_t +  \int_{\Pi^d} \rho_t \delta_t \cdot \nabla \log \frac{\rho_t }{\bar \rho_t }. 
\end{equation}

Recall that we  write the kernel $K= (K_1, \cdots, K_d)$ and its component $K_i = \sum_{j=1}^d \partial_{x_j} U_{ij}(x)$, where $U= (U_{ij})_{1 \leq i, j \leq d}$ is a matrix-valued potential function for instance can be defined as in  \eqref{eqn_K_as_divergence}. 
 Consequently 
\[
- \int \rho_t K * (\rho_t - \bar \rho_t) \cdot \nabla \log \bar \rho_t = - \sum_{i, j=1}^d \int \rho_t \partial_{x_j} U_{ij} * (\rho_t - \bar \rho_t) \partial_{x_i} \log \bar \rho_t, 
\]
which equals to 
\[
\sum_{i, j=1}^d  \int U_{ij}* (\rho_t - \bar \rho_t) \partial_{x_j}\big( \frac{\rho_t}{\bar \rho_t} \partial_{x_i} \bar \rho_t \big) = A + B 
\]
by integration by parts, 
where further 
\[
A= \sum_{i, j=1}^d  \int V_{i  j} * (\rho_t - \bar \rho_t) \, \partial_{x_i} \bar \rho_t \,  \partial_{x_j} \, \frac{\rho_t}{\bar \rho_t} = \int U * (\rho_t - \bar \rho_t) : \nabla \bar \rho_t \otimes \nabla \frac{\rho_t }{\bar \rho_t}, 
\]
and 
\[
B = \sum_{i, j=1}^d \int \rho_t U_{ij} *(\rho_t - \bar \rho_t) \frac{\partial_{x_i x_j}^2 \bar \rho_t }{\bar \rho_t} = \int \rho_t U * (\rho_t - \bar \rho_t) : \frac{\nabla^2 \bar \rho_t}{\bar \rho_t}. 
\]
Noticing that $\nabla \frac{\rho_t}{\bar \rho_t} =\frac{\rho_t}{\bar \rho_t} \nabla \log \frac{\rho_t }{\bar \rho_t }$, one estimates $A$ as follows 
\[
\begin{split}
	A & = \int  \rho_t U * (\rho_t - \bar \rho_t) : \nabla \log \bar \rho_t \otimes \nabla \log  \frac{\rho_t }{\bar \rho_t} \\
	&  \leq \frac{\nu}{4} \int \rho_t |\nabla \log \frac{\rho_t }{\bar \rho_t}|^2 + \frac{1}{\nu} \int \rho_t | (\nabla \log \bar \rho_t)^\top  \, U*(\rho - \bar \rho) |^2 \\
	& \leq \frac{\nu}{4} \int \rho_t |\nabla \log \frac{\rho_t }{\bar \rho_t}|^2 + \frac{1}{\nu } \|U\|_{L^\infty}^2 \|\nabla \log \bar \rho_t\|_{L^\infty}^2  \|\rho_t - \bar \rho_t\|_{L^1}^2,  
\end{split} 
\]
and again by  Csisz\'ar–Kullback–Pinsker inequality, one has that 
\[
A \leq \frac{\nu}{4} \int \rho_t |\nabla \log \frac{\rho_t }{\bar \rho_t}|^2 + \frac{2}{\nu } \|U\|_{L^\infty}^2 \|\nabla \log \bar \rho_t\|_{L^\infty}^2 \int \rho_t \log \frac{\rho_t}{\bar \rho_t}. 
\]

Now it only remains to control $B$. Recall the following famous Gibbs inequality 
\begin{lemma}[Gibbs inequality]\label{Gibbs} For any parameter $\eta > 0$, and  probability measures $\rho, \bar \rho \in \mathcal{P}(\mathcal{X}) \cap L^1(\mathcal{X})$, and $\phi$  a real-valued function defined on $\mathcal{X}$,  one has the following change of reference measure inequality 
	\[
	\int_{\mathcal{X}} \rho(x) \phi(x) \ud x    \leq \frac{1}{\eta } \Big( \int_{\mathcal{X}} \rho(x) \log \frac{\rho(x)}{\bar \rho(x) } \ud x  + \log \int_{\mathcal{X}} \bar \rho(x) \exp (\eta \phi (x))  \ud x \Big). 
	\]
	
\end{lemma} 
The proof of this inequality can be found in section 13.1 in  \citep{erdHos2017dynamical}. 

To control $B$, we write that $\phi= U * (\rho_t - \bar \rho_t) : \frac{\nabla^2 \bar \rho_t}{\bar \rho_t}$ and thus  $B = \int \rho_t  \phi$.  We choose a positive parameter $\eta>0$ such that 
\[
\frac{1}{\eta} = 2 \|U \|_{L^\infty} \Big\| \frac{\nabla^2 \bar \rho_t}{\bar \rho_t} \Big\|_{L^\infty}. 
\]
Now we apply Lemma \ref{Gibbs} to obtain that 
\[
B = \int \rho_t \phi \leq \frac{1}{\eta} \left( \int \rho_t  \log \frac{\rho_t }{\bar \rho_t} + \log \int \bar \rho_t \exp(\eta \phi )\right). 
\]
Note that $\eta >0$ is chosen so small such that  
\[
\begin{split}
	\eta \| \phi\|_{L^\infty} &  \leq  \frac{1}{2 \|U \|_{L^\infty} \Big\| \frac{\nabla^2 \bar \rho_t}{\bar \rho_t}\Big\|_{L^\infty}}  \|U\|_{L^\infty}  \|\rho_t - \bar \rho_t\|_{L^1}  \Big\| \frac{\nabla^2 \bar \rho_t}{\bar \rho_t}\Big\|_{L^\infty}   \\
	& \leq \frac 1 2 \|\rho_t - \bar \rho_t \|_{L^1} \leq 1, 
\end{split} 
\]
since for two probability densities it always holds $\|\rho_t - \bar \rho_t \|_{L^1} \leq 2$. Consequently, applying the inequality $\exp(x) \leq 1 + x + \frac e 2 x^2$ for $|x|\leq 1$, we have 
\[
\int \bar \rho_t \exp(\eta \phi ) \leq \int \bar \rho_t \left(  1 + \eta \phi + \frac e 2 \eta^2 \phi^2 \right) \leq 1 + 0 + \frac{e}{2}  \Big( \frac 1 2 \| \rho_t - \bar \rho_t\|_{L^1} \Big)^2  \leq 1 + \frac{e}{4} \mathbf{KL}(\rho_t \vert \bar \rho_t), 
\]
where 
\[
\int \bar \rho_t \phi = \int U* (\rho_t - \bar \rho_t) : \nabla^2 \bar \rho_t =\int  \sum_{i, j=1}^d \partial_{x_i x_j}U * (\rho_t - \bar \rho_t) \bar \rho_t = \int \udiv K *(\rho_t - \bar \rho_t ) \bar \rho_t = 0, 
\]
since $\udiv K =0$. 

To sum it up, in particular since $\log (1 + x ) \leq x $ for $x > 0$, one has 
\[
B \leq \frac{1}{\eta} \Big(  1 + \frac e 4 \Big) \mathbf{KL}(\rho_t \vert \bar \rho_t ) \leq 4  \|U \|_{L^\infty} \Big\| \frac{\nabla^2 \bar \rho_t}{\bar \rho_t} \Big\|_{L^\infty} \mathbf{KL}(\rho_t \vert \bar \rho_t ). 
\]

Combining \eqref{KL_NS_Evo}, the estimates for $A$ and $B$, one has 
\begin{equation}\label{NSEFinal}
	\begin{split}
		& \frac{\ud }{\ud t} \int \rho_t \log \frac{\rho_t}{\bar \rho_t} \leq - \frac{3 \nu  }{4}  \int \rho_t |\nabla \log \frac{\rho_t}{\bar \rho_t}|^2 + M(t) \int \rho_t \log \frac{\rho_t }{\bar \rho_t} + \int \rho_t \delta_t \cdot \nabla \log \frac{\rho_t}{\bar \rho_t} \\
		& \leq - \frac \nu  2 \int \rho_t |\nabla \log \frac{\rho_t}{\bar \rho_t}|^2 +M(t) \int \rho_t \log \frac{\rho_t }{\bar \rho_t} + \frac{1}{\nu } \int \rho_t |\delta_t|^2 
	\end{split}
\end{equation}
where 
\[
M(t) = \frac{2}{\nu } \|U\|_{L^\infty}^2 \|\nabla \log \bar \rho_t\|_{L^\infty}^2 + 4  \|U \|_{L^\infty} \Big\| \frac{\nabla^2 \bar \rho_t}{\bar \rho_t} \Big\|_{L^\infty} = M(t; \nu, U, \bar \rho_t). 
\]
Since the matrix-valued potential function $U$ is bounded ($\|U(\vx)\|_{op}\leq1/4$ when $U$ takse the form (\ref{eqn_K_as_divergence})), and under suitable assumptions for the initial data $\bar \rho_0$ (for instance $\bar \rho_0 \in C^3$ and there exists $c>1$ s.t. $\frac 1 c \leq \bar \rho \leq c $), one can obtain $\sup_{t \in [0, T] } M(t) \leq M < \infty$. We  recall Theorem 2 in  \citep{guillin2021uniform} as below for completeness. 

\begin{theorem} Given the initial data $\bar \rho_0 \in C^\infty (\Pi^d)$, such that there exists $c>1$, $\frac{1}{c} \leq \bar \rho_0 \leq c$. Then the vorticity formulation of the 2D Navier-Stokes equation 
	\[
	\partial_t \bar\rho_t + \udiv (\bar \rho_t K * \bar \rho_t ) = \nu \Delta \bar \rho_t, \quad \bar \rho(0, x) = \bar \rho_0(x), 
	\]
	has a unique bounded solution $\bar \rho(t, x) \in C^\infty([0, \infty) \times \Pi^d)$, and for any $t>0$,  for any $x \in \Pi^d$,  it holds that $\frac{1}{c} \leq \bar \rho(t, x) \leq c$. 
	
\end{theorem}

Finally,  we simplify \eqref{NSEFinal} to obtain that 
\[
\frac{\ud }{\ud t} \int \rho_t \log \frac{\rho_t}{\bar \rho_t} 
\leq  M \int \rho_t \log \frac{\rho_t }{\bar \rho_t} + \frac{1}{\nu } \int \rho_t |\delta_t|^2, 
\]
where $M = \sup_{t \in [0, T] } M(t; \nu, U, \bar \rho_t)< \infty$. By Gronwall inequality, one finally obtains that 
\[
\sup_{t \in [0, T]} \int_{\Pi^d} \rho_t \log \frac{\rho_t}{\bar \rho_t} \ud x \leq \frac{1}{\nu} \exp(M T ) R(\theta). 
\]

As noted in \citep{guillin2021uniform},in particular Corollary 2 there, one can improve the above time-dependent estimate ($\exp(MT)$) to uniform-in-time estimate by using Logarithmic Sobolev inequality. Indeed, given that $\frac 1 c \leq \bar \rho_t \leq c$, one has that
\begin{equation}
    \label{LogSoboIne}
    \int_{\Pi^d} \rho_t \log \frac{\rho_t}{\bar \rho_t} \ud x  \leq \frac{c^2}{8 \pi^2 } \int_{\Pi^d} \rho_t |\nabla_x \log \frac{\rho_t}{\bar \rho_t }|^2 \ud x. 
 \end{equation}
 Combining \eqref{LogSoboIne} and  \eqref{NSEFinal}, one obtains that 
 \[
\frac{\ud }{\ud t} \mathbf{KL}(\rho_t \vert \bar \rho_t) \leq \Big( M(t) - \frac{4 \pi^2 \nu  }{c^2} \Big) \mathbf{KL} (\rho_t \vert \bar \rho_t) +  \frac{1}{\nu } \int \rho_t |\delta_t|^2. 
 \]
Multiplying the factor $\exp(\frac{4 \pi^2 \nu  }{c^2} t - \int_0^t M(s) \ud s)$ and noting in particular $\mathbf{KL} (\rho_0 \vert \bar \rho_0) =0$,  one obtains that 
\[
\mathbf{KL}(\rho_t \vert \bar \rho_t) \leq \int_0^t \exp(\frac{4 \pi^2 \nu  }{c^2} (s- t) + \int_s^t M(u) \ud u ) f(s) \ud s. 
\]
Indeed, under the assumptions as in Theorem \ref{NSMainEstimate}, one has that there exists a universal $C>0$, such that 
\[
\int_0^\infty M(t) \ud t = C < \infty. 
\]
We thus immediately obtain that 
\[
\sup_{t \in [0, T] }\mathbf{KL}(\rho_t \vert \bar \rho_t) \leq  \frac{e^C}{\nu} \int_0^T \int_{\Pi^d} \rho_t |\delta_t|^2 \ud x \ud t. 
\]
This completes the proof of Theorem \ref{NSMainEstimate}.


\paragraph{The McKean-Vlasov PDEs, i.e. \eqref{eqn_MVE}, with bounded interactions $K \in L^\infty$ }  
As mentioned in the main body of this article,  it is much easier to  obtain the stability estimate for the McKean-Vlasov PDE with bounded interactions. 
\begin{theorem}[Stability Estimate for McKean-Vlasov PDE with $K \in L^\infty$] \label{theorem_bounded_K}
	Assume that $K \in L^\infty$. One has the estimate that 
	\[
	\sup_{t \in [0, T]} \mathbf{KL}(\rho_t^f \vert \bar \rho_t ) \leq   \frac{1}{\nu}\exp\Big(\frac{2 \|K \|_{L^\infty}^2}{\nu } T  \Big)  R(f), 
	\]
	where we recall the self-consitency potential/loss function $R(\theta)$ reads 
	\[
	R(f) = \int_0^T \int_{\X} |f(t, x) + \nabla V (x) - K * \rho_t^f + \nu \nabla \log \rho_t^\theta  |^2 \ud \rho_t^f(x) \ud t. 
	\]
\end{theorem}

\begin{proof}
	Here we give the control of the growth of the KL divergence for systems with bounded kernels.  
	Applying Cauchy-Schwarz inequality twice for the entropy dissipation terms  in Lemma \ref{TimeEvolKLMV} to obtain 
	\[
	\int_{\Pi^d} \rho_t K * (\rho_t - \bar \rho_t ) \cdot \nabla  \log  \frac{\rho_t }{\bar \rho_t}  \leq \frac{\nu}{4} \int \rho_t |\nabla \log \frac{\rho_t}{\bar \rho_t}|^2 + \frac{1}{\nu } \int \rho_t |K * (\rho_t - \bar \rho_t)|^2, 
	\]
	and 
	\[
	\int_{\Pi^d} \rho_t \delta_t \cdot \nabla \log \frac{\rho_t }{\bar \rho_t } \leq \frac{\nu}{4} \int \rho_t |\nabla \log \frac{\rho_t}{\bar \rho_t}|^2 + \frac{1}{\nu } \int \rho_t |\delta_t|^2. 
	\]
	Furthermore, 
	\[
	\int \rho_t |K * (\rho_t - \bar \rho_t)|^2 \leq \|K \|_{L^\infty}^2  \|\rho_t - \bar \rho_t\|_{L^1}^2 \leq  2 \|K \|_{L^\infty}^2  \int \rho_t \log \frac{\rho_t}{\bar \rho_t }, 
	\]
	where the last inequality is simply the Csisz\'ar–Kullback–Pinsker inequality \citep{villani2009optimal}. Combining the above estimates, we obtain that given that $K \in L^\infty$, 
	\[
	\frac{\ud }{\ud t} \int_{\Pi^d} \rho_t \log \frac{\rho_t }{\bar \rho_t} = - \frac{\nu}{2}   \int_{ \Pi^d}\rho_t |\nabla \log \frac{\rho_t}{\bar \rho_t}|^2 + \frac{2 \|K \|_{L^\infty}^2}{\nu } \int \rho_t \log \frac{\rho_t}{\bar \rho_t } + \frac{1}{\nu } \int \rho_t |\delta_t|^2. 
	\]
	Currently, we are not interested in the long time behavior, so we first ignore the negative term above to obtain that 
	\[
	\frac{\ud }{\ud t} \int_{\Pi^d} \rho_t \log \frac{\rho_t }{\bar \rho_t}  \leq    \frac{2 \|K \|_{L^\infty}^2}{\nu } \int \rho_t \log \frac{\rho_t}{\bar \rho_t } + \frac{1}{\nu } \int \rho_t |\delta_t|^2. 
	\]
	By Gronwall inequality, we obtain that
	\[
	\int_{\Pi^d} \rho_t \log \frac{\rho_t }{\bar \rho_t} \leq  \frac{1}{\nu}\exp\Big(\frac{2 \|K \|_{L^\infty}^2}{\nu } t \Big)   \int_0^t \int \rho_s |\delta_s|^2 \ud x  \ud s. 
	\]
	
	\end{proof}

\subsection{The McKean-Vlasov equation with Coulomb interactions}

\begin{proof}[Proof of Theorem \ref{ThmCoul}]  We first prove the case when $\nu >0$. Applying Cauchy-Schwarz inequality to the right-hand side of $\frac{\ud }{\ud t} E(\rho_t^f, \bar \rho_t)$ in Lemma \ref{TimeEvoMFE},  one has  
\[
\begin{split}
	\frac{\ud }{\ud t } E(\rho_t^f, \bar \rho_t)  & \leq  \frac 1 2  \int_{\mathcal{X}} \rho_t^f \, |\delta_t |^2 \ud x  \\
	&  - \frac{1}{2} \int_{\mathcal{X}^2} K(x-y) \cdot \Big( \mathcal{A}[\bar \rho_t](x) - \mathcal{A}[\bar \rho_t](y) \Big) \ud (\rho_t^f - \bar \rho_t )^{\otimes 2 }(x, y). 
\end{split}
\]
By Lemma 5.2 in \cite{bresch2019modulated}, as long as the ground truth ``velocity field" $\mathcal{A}[\bar \rho_t]$ is Lipschitz, i.e. $\mathcal{A}[\bar \rho] \in W^{1, \infty}$, or equivalently  $\nabla^2 V \in W^{1, \infty}, \nabla^2 \log \bar \rho_t \in L^\infty, K * \bar \rho_t \in W^{1, \infty}$, using the particular structure introduced by the Coulomb interactions (note that $- \Delta g = \delta_0$ and $K = - \nabla g$), we have the estimate 
\[
\begin{split}
	&  - \frac{1}{2} \int_{\mathcal{X}^2} K(x-y) \cdot \Big( \mathcal{A}[\bar \rho_t](x) - \mathcal{A}[\bar \rho_t](y) \Big) \ud (\rho_t^f - \bar \rho_t )^{\otimes 2 }(x, y) \\
	& \leq C \|\nabla \mathcal{A}[\bar \rho_t]\|_{L^\infty}  F(\bar \rho_t^f, \bar \rho_t).   \\
\end{split} 
\] 
This estimate can be obtained either by Fourier method  \citep{bresch2019modulated} or by the stress-energy tensor approach as in \citep{serfaty2020mean}. 
We emphasize that those assumptions made on $(\bar \rho_t)_{t \in [0, T]}$  can be obtained by propagating similar conditions on the initial data $\bar \rho_0$. 
This estimate actually holds for more general choices of $g$ or  $K$. See more examples including Riesz kernels in \citep{bresch2019modulated}.  Moreover, the Lipschitz regularity of $\mathcal{A}[\bar \rho_t]$ can also be relaxed a bit. See for instance in \citep{rosenzweig2022mean}.

Combining  previous two estimates, one has 
	\[
	\frac{\ud }{\ud t } E(\rho_t^f, \bar \rho_t)   \leq  \frac 1 2  \int_{\mathcal{X}} \rho_t^f \, |\delta_t |^2 \ud x  + C C_1  F(\rho_t^f, \bar \rho_t) \leq \frac 1 2  \int_{\mathcal{X}} \rho_t^f \, |\delta_t |^2 \ud x +  C C_1   E(\rho_t^f, \bar \rho_t). 
	\]
	Then applying Gronwall inequality concludes the proof of the case when $\nu >0$.
	
   Now we prove the deterministic case when $\nu=0$. Now the relative entropy or KL divergence does not play a role since there is no Laplacian term in \eqref{eqn_MVE}.  Lemma \ref{ModuEnergyEvo} now reads 
   	\[
   \begin{split}
   	\frac{\ud }{\ud t } F(\rho_t^f, \bar \rho_t)&  =  - \int_{\mathcal{X}} \rho_t^f |K *(\rho_t^f - \bar \rho)|^2 - \int_{\mathcal{X}} \rho_t^f \, \delta_t \cdot K * (\rho_t^f - \bar \rho_t ) \\
   	&  - \frac{1}{2} \int_{\mathcal{X}^2} K(x-y) \cdot \Big( \mathcal{A}[\bar \rho_t](x) - \mathcal{A}[\bar \rho_t](y) \Big) \ud (\rho_t^f - \bar \rho_t )^{\otimes 2 }(x, y). \\
      \end{split}
   \]
   Applying Cauchy-Schwarz to the 2nd term in the right-hand side above, we obtain that 
   \[
    \begin{split}
    	\frac{\ud }{\ud t } F(\rho_t^f, \bar \rho_t) \leq    \frac 1 2 \int_{\mathcal{X}} \rho_t^f \, |\delta_t|^2     - \frac{1}{2} \int_{\mathcal{X}^2} K(x-y) \cdot \Big( \mathcal{A}[\bar \rho_t](x) - \mathcal{A}[\bar \rho_t](y) \Big) \ud (\rho_t^f - \bar \rho_t )^{\otimes 2 }(x, y). \\
    \end{split}
   \]
   Again assuming that the ``velocity field'' $\mathcal{A}[\bar \rho_t](\cdot)$ is Lipschitz will give us 
   \[
   \frac{\ud }{\ud t } F(\rho_t^f, \bar \rho_t) \leq  \frac 1 2 \int_{\mathcal{X}} \rho_t^f \, |\delta_t|^2  + C C_1 F(\rho_t^f, \bar \rho_t). 
   \] 
   Applying Gronwall inequality again conclude all the proof.

\end{proof}

\clearpage
\clearpage
\section{Approximation Error of Neural Network} \label{appendix_approximation_error_NN}
We show that in a function class $\gF$ with sufficient capacity, there exists at least one element $\hat f\in \gF$ such that $R(\hat f)$ is small.
In particular, we are interested in the function class of neural networks.

We will focus on the case where the domain is the torus $\X = \Pi^d$, i.e. a $d$ dimensional box with size $L$ endowed with the periodic boundary condition. For the simplicity of notations, we denote the underlying velocity by $\bar f_t = \gA[\bar \rho_t]$, where the operator $\gA$ is defined in \eqref{eqn_operator_A}.

In the following, we focus on the Coulomb case where $K$ is defined in \eqref{eqn_coulomb_interaction}. The Biot-Savart case (\ref{eqn_nse}) can be treated similarly.
\begin{proof}[Proof of Theorem \ref{thm_approximation_error_NN}]
	In \eqref{eqn_pathwise_reformulation}, we showed that for any hypothesis velocity $f$, the \EINN\ loss $R(f)$ admits the trajectory-wise reformulation:
	\begin{equation}
		R(f) = \int_\X \ud \vx_0 \bar \rho_0(\vx) \int_0^T \ud t\|\delta^f_t\circ X^f_t(\vx_0)\|^2,
	\end{equation}
	where we recall the definition of $\delta^f_t$ in \eqref{eqn_perturbation}. 
    Note that, as a general principle, in this proof, we will use the superscript to emphasize the dependence on a velocity $f$, e.g. the flow map $X^f_t$.
	
	From Assumption \ref{ass_appendix_approximation} we know that there exists $\hat f\in\gF$ such that $\|\bar f - \hat f\|_{W^{2,\infty}(\X)} \leq \epsilon$.
	In the following, we show that $R(\hat f)$ is small.
	
	Define 
	\begin{equation}
		A^f_\vx(t) \defi \int_0^t \|\delta_s^f\circ X_s^f(\vx)\|^2 \ud s.
	\end{equation}
	We have
	\begin{equation}
		R(f) = \int_\X A_\vx^f(T) \ud \bar \rho_0(\vx).
	\end{equation}
	Recall that $\bar f$ denotes the underlying velocity and hence $\delta_t^{\bar f} \equiv 0$ and $\rho_t^{\bar f} \equiv \bar \rho_t$,
	where we recall that $\rho_t^{\bar f}$ is the solution to the continuity equation (\ref{eqn_CE}) with velocity field $\bar f$.
	We can bound
	\begin{align}
		\notag \frac{\partial}{\partial t}A^{\hat f}_x(t) =&\ \|\delta_t^{\hat f}\circ X_t^{\hat f}(\vx)\|^2 = \|\delta_t^{\hat f}\circ X_t^{\hat f}(\vx) - \delta_t^{\bar f}\circ X_t^{\bar f}(\vx)\|^2 \\
		\notag \leq&\ 4\|\left(\hat f_t\circ X_t^{\hat f} - \bar f_t\circ X_t^{\bar f}\right)(\vx)\|^2 + 4\|\left(\nabla V\circ X_t^{\hat f} - \nabla V\circ X_t^{\bar f}\right)(\vx)\|^2 \\
		\notag &\ + 4\|\left((K\ast \rho_t^{\hat f})\circ X_t^{\hat f} - (K\ast \bar \rho_t)\circ X_t^{\bar f}\right)(\vx)\|^2 + 4\nu^2\|\left(\nabla \log \rho_t^{\hat f}\circ X_t^{\hat f} - \nabla \log \bar \rho_t\circ X_t^{\bar f}\right)(\vx)\|^2\\
		= &\ \textcircled{1} + \textcircled{2} + \textcircled{3} + \textcircled{4}. \label{eqn_appendix_decomposition_A}
	\end{align}
	We will bound each term on the R.H.S. individually.
	The following lemmas will be useful:
	\begin{lemma} \label{lemma_appendix_Lipschitz_X_t_f}
		For two Lipschitz continuous velocity field $f_1, f_2 \in \gC^1(\X)$, we have for any $t\in[0, T]$
		\[
		\|X_t^{f_1}(\vx) - X_t^{f_2}(\vx)\|^2 \leq A_1(T)\|f_1 - f_2\|^2_{\gL^\infty(\X)}.
		\]
	\end{lemma}
    \begin{proof}
        Denote $\vx^i(t) = X_t^{f_i}(\vx_0)$ for $i = 1, 2$.
        \begin{align*}
            \frac{\ud }{\ud t}\|\vx^1(t) - \vx^2(t)\|^2 \leq&\ \|\vx^1(t) - \vx^2(t)\|^2 + \|f_1(t, \vx^1(t)) - f_2(t, \vx^2(t))\|^2\\
            \leq &\ C \|\vx^1(t) - \vx^2(t)\|^2 + \|f_1(t, \vx^2(t)) - f_2(t, \vx^2(t))\|^2 \\
            \leq &\ C \|\vx^1(t) - \vx^2(t)\|^2 + \|f_1 - f_2\|^2_{\gL^\infty(\X)}.
        \end{align*}
        Using the Gr\"onwall's inequality, we have the result.
    \end{proof}
	\begin{lemma} \label{lemma_appendix_Lipschitz_X_t}
		Suppose that $f\in \gC^1$ is Lipschitz continuous. We have that $X_t^f$ is an $A_2(T)$-Lipschitz continuous map. 
		For $f \in \gL^\infty(\X)$, we have $\|X_t^f(\vx)\|^2 \leq \|x\|^2 + t\|f\|^2_{\gL^\infty}$.
	\end{lemma}
    \begin{proof}
        Denote $\vx^i(t) = X_t^{f}(\vx^i_0)$ for $i = 1, 2$.
        \begin{align*}
            \frac{\ud }{\ud t}\|\vx^1(t) - \vx^2(t)\|^2 \leq&\ \|\vx^1(t) - \vx^2(t)\|^2 + \|f(t, \vx^1(t)) - f(t, \vx^2(t))\|^2\\
            \leq &\ C\|\vx^1(t) - \vx^2(t)\|^2.
        \end{align*}
        Using the Gr\"onwall's inequality, we have that $X_t^f$ is Lipschitz continuous.
    \end{proof}
	\begin{lemma}
		For $f \in \gC^2(\X)$, suppose that $\nabla (\udiv  f)\in \gL^\infty(\X)$ and $\gJ_f \in \gL^\infty(\X)$. 
        Further suppose that the initial distribution $\bar \rho_0$ satisfies $\nabla \log \bar \rho_0 \in \gL^\infty(\X)$.
        We have that $\nabla \log \rho_t^f\circ X_t^f \in \gL^\infty(\X)$.
	\end{lemma}
    \begin{proof}
        Denote $\vx(t) = X_t^{f}(\vx_0)$. From \eqref{eqn_dynamics_of_score}, we have
        \begin{equation}
           \frac{\ud }{\ud t} \|\nabla \log \rho_t^f(\vx(t))\|^2 \leq C (1+\|\nabla \log \rho_t^f(\vx(t))\|^2).
        \end{equation}
        Using the Gr\"onwall's inequality, we have $\|\nabla \log \rho_t^f(\vx(t))\|^2 < \infty$ for $\vx_0 \in \X$.
    \end{proof}
	\paragraph{Bounding \textcircled{1} in \eqref{eqn_appendix_decomposition_A}}
	We have
	\begin{align*}
		\notag &\ \|\left(\hat f_t\circ X_t^{\hat f} - \bar f_t\circ X_t^{\bar f}\right)(\vx)\| \\
		\leq &\ \|\left(\hat f_t\circ X_t^{\hat f} - \bar f_t\circ X_t^{\hat f}\right)(\vx)\| + \|\left(\bar f_t\circ X_t^{\hat f} - \bar f_t\circ X_t^{\bar f}\right)(\vx)\| \\
		\leq &\ \epsilon + \epsilon \cdot \| X_t^{\hat f}(\vx) - X_t^{\bar f}(\vx)\| \leq C_1(T) \epsilon.
	\end{align*}
	\paragraph{Bounding \textcircled{2} in \eqref{eqn_appendix_decomposition_A}}
	We have from Assumption \ref{ass_appendix_approximation} and Lemma \ref{lemma_appendix_Lipschitz_X_t_f}
	\begin{align*}
		\|\left(\nabla V\circ X_t^{\hat f} - \nabla V\circ X_t^{\bar f}\right)(\vx)\| \leq C_2(T) \epsilon.
	\end{align*}
	\paragraph{Bounding \textcircled{3} in \eqref{eqn_appendix_decomposition_A}}
	Bounding \textcircled{3} in \eqref{eqn_appendix_decomposition_A} requires a more sophisticated analysis which is the major technical challenge of this proof. 
	For the simplicity of notations, for a fixed $x$ and $y$, denote $\hat \vx(t) = X_t^{\hat f}(\vx)$, $\bar \vx(t) =  X_t^{\bar f}(\vx)$, $\hat \vy(t) = X_t^{\hat f}(\vy)$, $\bar \vy(t) =  X_t^{\bar f}(\vy)$. For any $\epsilon'$ which is to be determined later, we have
	\begin{align}
		\notag &\ \|\left((K\ast \rho_t^{\hat f})\circ X_t^{\hat f} - (K\ast \bar \rho_t)\circ X_t^{\bar f}\right)(\vx)\|^2
		= \| \int_\X K(\hat \vx(t) - \hat \vy(t)) - K(\bar \vx(t) - \bar \vy(t)) \ud \bar \rho_0(\vy)\|^2\\
		\tag{\textcircled{A}}\leq &\ 2\| \int_{\|\hat \vx(t) - \hat \vy(t)\|\leq \epsilon'} K(\hat \vx(t) - \hat \vy(t)) - K(\bar \vx(t) - \bar \vy(t)) \ud \bar \rho_0(\vy)\|^2\\
		\tag{\textcircled{B}} &\ + 2\| \int_{A_2(T)L\geq \|\hat \vx(t) - \hat \vy(t)\|\geq \epsilon'} K(\hat \vx(t) - \hat \vy(t)) - K(\bar \vx(t) - \bar \vy(t)) \ud \bar \rho_0(\vy)\|^2.
	\end{align}
	Note that the upper bound on $\|\hat \vx(t) - \hat \vy(t)\|$ in \textcircled{B} comes from Lemma \ref{lemma_appendix_Lipschitz_X_t} and the facts that $\X = \Pi^d$ is bounded with size $L$.
	To bound \textcircled{A}, we have
	\begin{align*}
		&\ \| \int_{\|\hat \vx(t) - \hat \vy(t)\|\leq \epsilon'} K(\hat \vx(t) - \hat \vy(t)) - K(\bar \vx(t) - \bar \vy(t)) \ud \bar \rho_0(\vy)\| \\
		\leq &\  \int_{\|\hat \vx(t) - \hat \vy(t)\|\leq \epsilon'} \|K(\hat \vx(t) - \hat \vy(t))\| + \|K(\bar \vx(t) - \bar \vy(t))\| \ud \bar \rho_0(\vy)\\
		= &\ \int_{\|\hat \vx(t) - \hat \vy(t)\|\leq \epsilon'} \frac{1}{\|\hat \vx(t) - \hat \vy(t)\|^{d-1}} + \frac{1}{\|\bar \vx(t) - \bar \vy(t)\|^{d-1}} \ud \bar \rho_0(\vy) = \textcircled{C} + \textcircled{D}.
	\end{align*}
	We can bound \textcircled{C} by
	\begin{align}
		&\ \int_{\|\hat \vx(t) - \hat \vy(t)\|\leq \epsilon'} \frac{1}{\|\hat \vx(t) - \hat \vy(t)\|^{d-1}} \ud \bar \rho_0(\vy) = \int_{\|\hat \vx(t) - y\|\leq \epsilon'} \frac{1}{\|\hat \vx(t) - y\|^{d-1}} \ud \rho^{\hat f}_t(\vy) \leq \|\rho^{\hat f}_t\|_{\infty}\cdot \epsilon',
	\end{align}
	where in the above inequality we remove the singular term by using transforming to the polar coordinate system.
	To bound \textcircled{D}, we pick $\epsilon' = dA_1(T)\epsilon$, so that Lemma \ref{lemma_appendix_Lipschitz_X_t_f} implies 
	\begin{equation*}
		\{ y\in\X | \|\hat \vx(t) - \hat \vy(t)\|\leq \epsilon'\} \subseteq \{y\in\X | \|\bar \vx(t) - \bar \vy(t)\|\leq \frac{d+2}{d}\epsilon'\},
	\end{equation*}
	and consequently
	\begin{equation}
		\int_{\|\hat \vx(t) - \hat \vy(t)\|\leq \epsilon'} \frac{1}{\|\bar \vx(t) - \bar \vy(t)\|^{d-1}} \ud \bar \rho_0(\vy) \leq \int_{\|\bar \vx(t) - \bar \vy(t)\|\leq 2\epsilon'} \frac{1}{\|\bar \vx(t) - \bar \vy(t)\|^{d-1}} \ud \bar \rho_0(\vy) \leq \frac{d+2}{d} \|\rho^{\hat f}_t\|_{\infty}\cdot \epsilon'.
	\end{equation}
	To bound \textcircled{B}, note that
	\begin{equation}
		\nabla K(\vx) = \frac{1}{\|x\|^{d+2}}(\|x\|^2 I - d\cdot x\otimes x) \Rightarrow \|\nabla K(\vx)\|\leq \frac{d}{\|x\|^{d}}.
	\end{equation}
	Denote $\vz(t) = \min(\|\hat \vx(t) - \hat \vy(t)\|, \|\bar \vx(t) - \bar \vy(t)\|)$.
	Recall the choice of $\epsilon' = dA_1(T)\epsilon$.
	Using Lemma \ref{lemma_appendix_Lipschitz_X_t_f}, we have that 
	\begin{equation*}
		\|\vz(t)\| \geq \frac{d-2}{d} \|\hat \vx(t) - \hat \vy(t)\|.
	\end{equation*}


	Using the triangle inequality, we have
	\begin{align*}
		&\ \| \int_{A_2(T)L \geq\|\hat \vx(t) - \hat \vy(t)\|\geq \epsilon'} K(\hat \vx(t) - \hat \vy(t)) - K(\bar \vx(t) - \bar \vy(t)) \ud \bar \rho_0(\vy)\| \\
		\leq &\ \int_{A_2(T)L \geq\|\hat \vx(t) - \hat \vy(t)\|\geq \epsilon'} \|K(\hat \vx(t) - \hat \vy(t)) - K(\bar \vx(t) - \bar \vy(t))\| \ud \bar \rho_0(\vy)\\
		\leq &\ 2\epsilon' d\int_{A_2(T)L \geq\|\hat \vx(t) - \hat \vy(t)\|\geq \epsilon'} \frac{1}{\|\vz(t)\|^{d}} \ud \bar \rho_0(\vy) \\
		\leq &\ 2\epsilon' d\int_{A_2(T)L \geq\|\hat \vx(t) - \hat \vy(t)\|\geq \epsilon'} (\frac{d}{d-2})^d\frac{1}{\|\hat \vx(t) - \hat \vy(t)\|^{d}} \ud \bar \rho_0(\vy) \\
		\leq &\ 2e\epsilon' d \|\rho^{\hat f}_t\|_{\infty} \int_{A_2(T)L \geq \|\hat \vx(t) - y\|\geq .5\epsilon'} \frac{1}{\|y\|^d} \ud y \\
		=&\ 2e\epsilon' d \|\rho^{\hat f}_t\|_{\infty} \ln (A_2(T)L/\epsilon').
	\end{align*}

	Combining the bounds of \textcircled{A} and \textcircled{B}, we have that 
	\begin{equation}
		\textcircled{3} \leq C_3(T) (\epsilon \ln\frac{1}{\epsilon})^2.
	\end{equation}
	\paragraph{Bounding \textcircled{4} in \eqref{eqn_appendix_decomposition_A}}
	Denote $\hat \vx(t) = X_t^{\hat f}(\vx)$ and $\bar \vx(t) =  X_t^{\bar f}(\vx)$.
	Define
	\begin{equation}
		B_\vx(t) \defi \|\left(\nabla \log \rho_t^{\hat f}\circ X_t^{\hat f} - \nabla \log \bar \rho_t\circ X_t^{\bar f}\right)(\vx)\|^2 = \|\nabla \log \rho_t^{\hat f}(\hat{\vx}(t)) - \nabla \log \bar \rho_t(\bar{\vx}(t))\|^2.
	\end{equation}
	Computing its dynamics
	\begin{equation}
		\frac{\ud }{\ud t} B_\vx(t) \leq B_\vx(t) + \|\frac{\ud }{\ud t}\left(\nabla \log \rho_t^{\hat f}(\hat{\vx}(t)) - \nabla \log \bar \rho_t(\bar{\vx}(t))\right)\|^2
	\end{equation}
	Recall \eqref{eqn_dynamics_of_score}. We have that 
	\begin{equation}
		\frac{\ud }{\ud t} \nabla \log \rho_t^{\hat f}(\hat \vx(t)) = - \nabla  \left(\nabla \cdot \hat f_{t}(\hat \vx(t))\right) -   \left(\gJ_{\hat f_{t}}(\hat \vx(t))\right)^\top \nabla \log \rho_{t}^{\hat f}(\hat \vx(t)),
	\end{equation}
	\begin{equation}
		\frac{\ud }{\ud t} \nabla \log \rho_t^{\bar  f}(\bar  \vx(t)) = - \nabla  \left(\nabla \cdot \bar  f_{t}(\bar  \vx(t))\right) -   \left(\gJ_{\bar  f_{t}}(\bar  \vx(t))\right)^\top \nabla \log  \rho_{t}^{\bar  f}(\bar  \vx(t)),
	\end{equation}
	and hence
	\begin{align*}
		&\ \|\frac{\ud }{\ud t}\left(\nabla \log \rho_t^{\hat f}(\hat{\vx}(t)) - \nabla \log \bar \rho_t(\hat{\vx}(t))\right)\|^2 \\
		\leq &\ 2\|\nabla  \left(\nabla \cdot \hat f_{t}(\hat \vx(t))\right) - \nabla  \left(\nabla \cdot \bar  f_{t}(\bar  \vx(t))\right)\|^2 
		 + 2\|\left(\gJ_{\hat f_{t}}(\hat \vx(t))\right)^\top \nabla \log \rho_{t}^{\hat f}(\hat \vx(t)) - \left(\gJ_{\bar  f_{t}}(\bar  \vx(t))\right)^\top \nabla \log  \rho_{t}^{\bar  f}(\bar  \vx(t))\|^2 \\
		= &\ \textcircled{E} + \textcircled{F}.
	\end{align*}
	We now bound these two terms individually. To bound \textcircled{E},
	\begin{align*}
		&\ \|\nabla  \left(\nabla \cdot \hat f_{t}(\hat \vx(t))\right) - \nabla  \left(\nabla \cdot \bar  f_{t}(\bar  \vx(t))\right)\| \\
		\leq &\ \|\nabla  \left(\nabla \cdot \hat f_{t}(\hat \vx(t))\right) - \nabla  \left(\nabla \cdot \hat  f_{t}(\bar  \vx(t))\right)\| + \|\nabla  \left(\nabla \cdot \hat f_{t}(\bar \vx(t))\right) - \nabla  \left(\nabla \cdot \bar  f_{t}(\bar  \vx(t))\right)\| \leq \epsilon + LA_1(T)\epsilon.
	\end{align*}
	To bound \textcircled{F}
	\begin{align*}
		&\ \|\left(\gJ_{\hat f_{t}}(\hat \vx(t))\right)^\top \nabla \log \rho_{t}^{\hat f}(\hat \vx(t)) - \left(\gJ_{\bar  f_{t}}(\bar  \vx(t))\right)^\top \nabla \log  \rho_{t}^{\bar  f}(\bar  \vx(t))\|^2\\
		\leq &\ 2\|\left(\left(\gJ_{\hat f_{t}}(\hat \vx(t))\right)^\top - \left(\gJ_{\bar  f_{t}}(\bar  \vx(t))\right)^\top\right) \nabla \log  \rho_{t}^{\hat  f}(\hat  \vx(t))\|^2 + 2\|\left(\gJ_{\bar f_{t}}(\bar \vx(t))\right)^\top \left(\nabla \log \rho_{t}^{\hat f}(\hat \vx(t)) - \nabla \log  \rho_{t}^{\bar  f}(\bar  \vx(t))\right)\|^2\\
		\leq &\ 2(1 + LA_1(T))^2\epsilon^2 + 2L^2 B_x(t).
	\end{align*}
	Consequently, using Gr\"onwall's inequality, we have that
	\begin{equation}
		\textcircled{4} \leq C_4(T)\epsilon^2.
	\end{equation}

	Combining all the estimations for \textcircled{1} to \textcircled{4}, we have that
	\begin{equation}
		R(\hat f) \leq C(T) \epsilon^2 (\ln \frac{1}{\epsilon})^2,
	\end{equation}
	for some constant $C(T)$ independent of $\epsilon.$
\end{proof}

\clearpage
\section{Discussion on the Unbounded Case} \label{appendix_unbounded}
In Section \ref{section_analysis}, we considered the torus case, i.e. $\X$ is a $d$-dimensional box with size $L$ with a periodic boundary condition. In this section, we consider the unbounded case, i.e. $\X = \sR^d$. 
There are two major differences: 
\begin{enumerate}[leftmargin=*]
    \item The first difference is that when $\X = \sR^d$, we would obtain an additional integral-of-divergence term from the operation of integration by parts. When $\X$ is a torus, using Gauss's divergence theorem and the periodic boundary condition, this term immediately vanishes, which simplifies the analysis. In contrast, for the unbounded case, we need to handle this term by assuming some additional regularity conditions.
    \item The second difference is that for the torus, it is reasonable to assume that the initial distribution $\bar \rho_0$ is fully supported, which is equivalent to the existence of some constant $c > 0$ such that $\bar \rho_0(\vx) \geq c$ for all $\vx\in\X$. Such an assumption will allow us to propagate the regularity of the initial distribution $\bar \rho_0$ to the solution at time $t$, i.e. $\bar \rho_t$.
    In contrast, for the unbounded case, such an assumption clearly does not hold since otherwise $\bar \rho_0$ would not be integrable. Consequently, we can no long propagate the regularity of the initial distribution and hence we need to directly make regularity assumptions on $\bar \rho_t$.
\end{enumerate}
In the following, we will focus on addressing the first point and provide sufficient conditions such that Lemmas \ref{TimeEvolKLMV} and  \ref{ModuEnergyEvo} can be recovered even in the unbounded case.
To elaborate a bit on the second point, the theorems that are derived in the main body of the submission remain valid under the regularity assumptions given therein. However, unlike the torus case, it is difficult to establish these regularity results for the unbounded case by assuming the regularity of the initial distribution $\bar \rho_0$.
\begin{lemma}[Analogy of Lemma  \ref{TimeEvolKLMV} in the unbounded case] \label{lemma_TimeEvolKLMV}
    Given the hypothesis velocity field $f=f(t, x) \in C^1_{t, x}$. Assume that $(\rho_t^f)_{t \in [0, T]}$ and $(\bar \rho_t)_{t \in [0, T]}$ are classical solutions to equation (\ref{eqn_CE}) and equation (\ref{eqn_MVE_CE}) respectively.  It holds that (recall the definition of $\delta_t$ in equation (\ref{eqn_perturbation})) 
    \begin{align*}
        \frac{\ud }{\ud t} \int_{\mathcal{X}} \rho^f_t \log \frac{\rho^f_t }{\bar \rho_t} =&\ - \nu \int_{ \X}\rho^f_t |\nabla \log \frac{\rho^f_t}{\bar \rho_t}|^2 +  \int_{\X} \rho^f_t K * (\rho^f_t - \bar \rho_t ) \cdot \nabla  \log  \frac{\rho^f_t }{\bar \rho_t} \\
        &\ +  \int_{\X} \rho^f_t \delta_t \cdot \nabla \log \frac{\rho^f_t }{\bar \rho_t } - \int \udiv \Big(\rho^f_t  (f_t  \log \frac{\rho^f_t}{\bar \rho_t} - \bar f_t)\Big).
    \end{align*}
	where $\X$ is the tours $\Pi^d$. All the integrands are evaluated at $\vx$. 
\end{lemma}
\begin{proof} 
	Recall the McKean-Vlasov  \eqref{eqn_MVE_CE} and the continuity  \eqref{eqn_CE_as_MVE}. For simplicity, we write that $\rho_t = \rho_t^f$ and $\bar f_t = \gA[\bar \rho_t]$. Then 
    \begin{align*}
        \frac{\ud }{\ud t } \int \rho_t \log \frac{\rho_t }{\bar \rho_t } =&\ - \int  \udiv \Big(    \rho_t  f_t \Big) \log \frac{\rho_t}{\bar \rho_t}  +  \int \frac{\rho_t}{\bar \rho_t}  \udiv \Big(   \bar \rho_t \bar f_t\Big)\\
        =&\ \int  \rho_t  f_t \nabla \log \frac{\rho_t}{\bar \rho_t}  -  \int \nabla \frac{\rho_t}{\bar \rho_t}  \bar \rho_t \bar f_t - \int \udiv \Big(    \rho_t  (f_t  \log \frac{\rho_t}{\bar \rho_t} - \bar f_t)\Big).
    \end{align*}
    We handle the first two terms on the R.H.S. just like the torus case and we can have the result.
\end{proof}

\begin{lemma}[Analogy of Lemma  \ref{ModuEnergyEvo} in the unbounded case] \label{lemma_ModuEnergyEvo}
Under the same assumptions as in Lemma \ref{TimeEvolKLMV}, given the diffusion coefficient $\nu \geq 0$, it holds that (recall the definition of $\delta_t$ in equation (\ref{eqn_perturbation})) 
    \begin{align*}
        \frac{\ud }{\ud t } F(\rho_t^f, \bar \rho_t)  =&\  - \int_{\mathcal{X}} \rho_t^f \|K *(\rho_t^f - \bar \rho_t)\|^2 - \int_{\mathcal{X}} \rho_t^f \, \delta_t \cdot K * (\rho_t^f - \bar \rho_t ) + \nu \int_{\mathcal{X}} \rho^f_t \, K * (\rho_t^f - \bar \rho_t )\cdot \nabla \log \frac{\rho_t^f}{\bar \rho_t} \\
		&\  - \frac{1}{2} \int_{\mathcal{X}^2} K(x-y) \cdot \Big( \mathcal{A}[\bar \rho_t](x) - \mathcal{A}[\bar \rho_t](y) \Big) \ud (\rho_t^f - \bar \rho_t )^{\otimes 2 }(x, y) \\
        &\ - \int \udiv \Big\{ g * (\rho^f_t - \bar \rho_t )(x)( \rho^f_t(x) f_t(x)   - \bar \rho_t(x) \bar f_t(x))\Big\} \ud x
    \end{align*}
	where we recall that the operator $\gA$ is defined in equation (\ref{eqn_operator_A}).
\end{lemma}
\begin{proof}
Recall that $K= - \nabla g$.  For simplicity, we write that $\rho_t = \rho_t^f$. Then 
	\[
	\begin{split} 
		&\ \frac{\ud }{\ud t } F(\rho_t, \bar \rho_t) = \frac{\ud }{\ud t } \frac{1 }{2} \int_{\mathcal{X}^2} g (x- y) \ud (\rho_t - \bar \rho_t)^{\otimes 2 } (x, y)  \\
		=&\ \int_\mathcal{X} g *(\rho_t - \bar \rho_t) (x) \big(\partial_t \rho_t (x) - \partial_t \bar \rho_t (x)  \Big) \ud x  \\
		=&\ - \int g * (\rho_t - \bar \rho_t ) (x) \udiv \Big\{  \rho_t(x) f_t(x)   - \bar \rho_t(x) \bar f_t(x)\Big\} \ud x \\
        =&\ \int \nabla g * (\rho_t - \bar \rho_t ) (x) \Big\{  \rho_t(x) f_t(x)   - \bar \rho_t(x) \bar f_t(x)\Big\} \ud x \\
        &\ - \int \udiv \Big\{ g * (\rho_t - \bar \rho_t )(x)( \rho_t(x) f_t(x)   - \bar \rho_t(x) \bar f_t(x))\Big\} \ud x
	\end{split} 
	\]
    We handle the first term on the R.H.S. just like the torus case and we can have the result.
\end{proof}

\subsection{Handling the Integral of the Divergence}
Given a vector field, the following lemma provides a sufficient condition for the volume integral of its divergence over $\X$ to be zero.
The idea is to construct a sequence of approximations to the integral of interest, each of which involves integration over a compact set. Consequently, Gauss's divergence theorem can be applied. We then utilize the dominant convergence theorem to exchange the order of the limit and integral.
\begin{lemma} \label{lemma_zero_divergence_integral_general}
    For a vector function $g: \sR^d \rightarrow\sR^d$ which satisfies 
    \begin{equation} \label{eqn_integrability_g}
        \int_{\sR^d} \ud\vx\ |\udiv\ g (\vx)| < \infty\quad \text{and}\quad \int_{\sR^d} \ud\vx\ \|g (\vx)\| <\infty,
    \end{equation}
    we have
    \begin{equation}
        \int_{\sR^d}\ud\vx\ \udiv\ g (\vx) = 0.
    \end{equation}
\end{lemma}
\begin{proof}
    Choose a cut-off function, indexed by $r > 1$, satisfying
    \begin{equation}
        \Phi_r(\vx) = \begin{cases}
            1, \quad& \text{if} \quad \|\vx\|\leq r, \\
            \frac{1}{2}(1+\cos(\pi \|\vx\| / r - 1)), & \text{if} \quad r<\|\vx\|\leq 2r,\\
            0, &\text{if} \quad 2r<\|\vx\|.
        \end{cases}
    \end{equation}
    We have $\|\nabla \Phi_r\|_{\gL^\infty} = O({1}/{r})$.
    Using the chain rule of divergence, we have that
    \begin{equation}
        \udiv_\vx(g \cdot \Phi_r) = \udiv_\vx(g)\cdot\Phi_r + g \cdot \nabla \Phi_r.
    \end{equation}
    We have $\int_{\sR^d}\ud \vx \udiv_\vx(g \Phi_r) (\vx) = 0$ for all $r$ and $\vx$, by noting $g \Phi_r(\vx) = 0$ for $\|\vx\| > 2r$ and using Gauss's divergence theorem on the $\vx$ variable.
    Using conditions (\ref{eqn_integrability_g}) and the dominated convergence theorem, we have
    \begin{align*}
        0 =&\ \lim_{r\rightarrow\infty} \int_{\X}\ud \vx\ [\udiv_\vx(g)\cdot\Phi_r] (\vx) + \lim_{r\rightarrow\infty} \int_{\X}\ud \vx\ [g \cdot \nabla \Phi_r] (\vx)\\
        =&\   \int_{\X}\ud \vx\ \lim_{r\rightarrow\infty} [\udiv_\vx(g)\cdot\Phi_r] (\vx) +  \int_{\X}\ud \vx\ \lim_{r\rightarrow\infty}[g \cdot \nabla \Phi_r] (\vx) \\
        =&\  \int_{\X}\ud \vx\ \udiv_\vx(g) (\vx),
    \end{align*}
    where in the last equality, we use $g \cdot \nabla \Phi_r (\vx) \leq \|g(\vx)\| \|\nabla \Phi_r (\vx)\| \rightarrow 0$ as $r\rightarrow\infty$.
\end{proof}
We now show that the divergence integrals in Lemmas \ref{lemma_TimeEvolKLMV} and \ref{lemma_ModuEnergyEvo} satisfy the requirements (\ref{eqn_integrability_g}), under the following regularity assumptions on the hypothesis velocity field $f$, initial distribution $\bar \rho_0$, and the ground truth solution $\bar \rho$.
\begin{assumption}\label{ass_regularity_force_field}
    $f \in \Lip(\X)$ and there exists some constant $L$, such that for all $t \in [0, T]$ and $\vx \in \X$ $\|[\nabla (\udiv f)](t, \vx)\| \leq L$.
\end{assumption}
\begin{assumption}\label{ass_regularity_initial_distribution}
    The initial distribution $\bar \rho_0$ satisfies
    \begin{equation}
        \int_{\X}\ud \vx\ \bar \rho_0(\vx) (|\log \bar \rho_0(\vx)|+1)(\|\vx\|+1)^{\alpha+1}(\|\nabla\log\bar \rho_0(\vx)\|+1) <\infty
    \end{equation}
\end{assumption}
\begin{assumption}\label{ass_regularity_ground_truth}
    Suppose that the ground truth $\bar \rho_t \in \gL^\infty$ is sufficiently regular such that
    \begin{equation}
        |\log \bar \rho_t(\vx)| + \|\nabla \log \bar \rho_t(\vx)\| \leq L(1 + \|\vx\|)^\alpha\ \text{and}\ \|\bar f_t(\vx)\| + |\udiv \bar f_t(\vx)|\leq L(1+\|\vx\|)^\alpha
    \end{equation}
    holds for all $\vx \in\X$ and $t \in [0, T]$ with some constant $\alpha$ and $L$. Here we denote $\bar f_t = \gA [\rho_t]$.
\end{assumption}
The following estimations of regularity will be helpful. The proof is deferred to the end of this section.
\begin{lemma}\label{lemma_regularity_estimation}
    Under Assumption \ref{ass_regularity_force_field}, we have the following estimations
    \begin{align*}
        \|\vx_t\|^2 \leq&\ \exp(t(1+2L^2))(\|\vx_0\|^2 + 1)\\
        |\log \rho_t^f(\vx_t)| \leq&\ |\log \bar \rho_0(\vx_0)| + L t\\
        \|\nabla \log \rho_t^f(\vx_t)\|^2 \leq&\ \exp(t(1+2L^2))(\|\nabla \log  \bar \rho_0(\vx_0)\|^2 + 1).
    \end{align*}
\end{lemma}
We now show that the integrals of the divergence in Lemmas \ref{lemma_zero_divergence_integral_KL} and \ref{lemma_zero_divergence_integral_modulated_energy} are zero.
\begin{lemma} \label{lemma_zero_divergence_integral_KL}
    Under Assumptions \ref{ass_regularity_force_field} to \ref{ass_regularity_ground_truth}, we have
    \begin{equation}\label{eqn_divergence_integral_KL}
        \int_{\X} \udiv \Big(\rho_t  (f_t  \log \frac{\rho^f_t}{\bar \rho_t} - \bar f_t)\Big) = 0.
    \end{equation}
\end{lemma}
\begin{proof}
    To establish Lemma \ref{lemma_zero_divergence_integral_KL}, we need to show that all the terms inside the divergence of \eqref{eqn_divergence_integral_KL} satisfy the integrability requirements (\ref{eqn_integrability_g}) in Lemma \ref{lemma_zero_divergence_integral_general}, which are handled one by one in the following.

    \begin{itemize}[leftmargin=*]
        \item We handle the term $\rho_t^f \log \rho_t^f f_t$.
        \begin{itemize}
            \item To show that $\int_{\X}\ud \vx\ \| \rho_t^f \log \rho_t^f f_t (\vx)\| <\infty$
            \begin{align*}
                &\ \int_{\X}\ud \vx\ \| [\rho_t^f \log \rho_t^f f_t] (\vx)\| = \int_{\X}\ud \vx\ \rho_t^f(\vx) \| [ \log \rho_t^f f_t] (\vx)\|\\
                =&\ \int_{\X}\ud \vx\ \bar \rho_0(\vx_0) \cdot|\log \rho_t^f(\vx_t)|\cdot\|f_t(\vx_t)\| \\
                \leq &\ \int_{\X}\ud \vx\ \bar \rho_0(\vx_0) \cdot(|\log \bar \rho_0(\vx_0)| + Lt)\cdot\exp(t(1+2L^2))(\|\vx_0\| + 1) < \infty.
            \end{align*}
    
            \item To show that $\int_{\X}\ud \vx\ | \udiv \left( \rho_t^f \log \rho_t^f f_t\right) (\vx)| <\infty$
            \begin{align*}
                &\ \udiv \left( \rho_t^f \log \rho_t^f f_t\right) = \udiv f_t \cdot \rho_t^f \log \rho_t^f + f_t \cdot \nabla (\rho_t^f \log \rho_t^f)\\
                =&\ \udiv f_t \cdot \rho_t^f\cdot \log \rho_t^f + (f_t \cdot \nabla \rho_t^f) \cdot \log \rho_t^f + (f_t \cdot \nabla\log \rho_t^f) \cdot\rho_t^f\\
                =&\ \rho_t^f\left(\udiv f_t \cdot \log \rho_t^f + (f_t \cdot \nabla \log \rho_t^f) \cdot (1+\log \rho_t^f)\right)
            \end{align*}
            We now bound
            \begin{align*}
                &\ \int_{\X}\ud \vx\ \rho_t^f(\vx) | [\udiv f_t \cdot \log \rho_t^f] (\vx)|\\
                =&\ \int_{\X}\ud \vx\ \bar \rho_0(\vx_0) | [\udiv f_t \cdot \log \rho_t^f] (\vx_t)| \leq \int_{\X}\ud \vx\ \bar \rho_0(\vx_0)\cdot L \cdot (tL + |\log \bar \rho_0(\vx_0)|) <\infty.
            \end{align*}
            and 
            \begin{align*}
                &\ \int_{\X}\ud \vx\ \rho_t^f(\vx) | [(f_t \cdot \nabla \log \rho_t^f) \cdot (1+\log \rho_t^f)] (\vx)| \\
                = &\ \int_{\X}\ud \vx\ \bar \rho_0(\vx_0) | [(f_t \cdot \nabla \log \rho_t^f) \cdot (1+\log \rho_t^f)] (\vx_t)| \\
                \leq &\ \int_{\X}\ud \vx\ \bar \rho_0(\vx_0) L  (1+\|\vx_0\|) \exp(t(1+2L^2))(\|\nabla\log\bar \rho_0(\vx_0)\|+1)(1+Lt+|\log \bar \rho_0(\vx_0)|) <\infty.
            \end{align*}
        \end{itemize}
        \item We handle the term $\rho_t^f \log \bar \rho_t f_t$.
        \begin{itemize}
            \item To show that $\int_{\X}\ud \vx\ \| \rho_t^f \log \bar \rho_tf_t (\vx)\| <\infty$
                \begin{align*}
                    &\ \int_{\X}\ud \vx\ \rho_t^f\|\log \bar \rho_tf_t (\vx)\| = \int_{\X}\ud \vx\ \bar \rho_0\|[\log \bar \rho_tf_t] (\vx_t)\| \\
                    \leq&\ \int_{\X}\ud \vx\ \bar \rho_0(\vx_0)|\log \bar \rho_t(\vx_t)|\|f_t(\vx_t)\| \leq \int_{\X}\ud \vx\ \bar \rho_0(\vx_0)L^2(1+\|\vx_t\|)^{\alpha+1}<\infty.
                \end{align*}
            \item To show that $\int_{\X}\ud \vx\ | \udiv \left( \rho_t^f \log \bar \rho_tf_t\right) (\vx)| <\infty$
            \begin{align*}
                &\ \udiv \left( \rho_t^f \log \bar \rho_tf_t\right) = \udiv f_t \cdot \rho_t^f \log \bar \rho_t+ f_t \cdot \nabla (\rho_t^f \log \rho_t)\\
                =&\ \udiv f_t \cdot \rho_t^f\cdot \log \bar \rho_t+ (f_t \cdot \nabla \rho_t^f) \cdot \log \bar \rho_t+ (f_t \cdot \nabla\log \rho_t) \cdot\rho_t^f\\
                =&\ \rho_t^f\left(\udiv f_t \cdot \log \bar \rho_t+ (f_t \cdot \nabla \log \rho_t^f) \log \bar \rho_t+ f_t \cdot \nabla\log \rho_t\right)
            \end{align*}
            We now bound 
            \begin{align*}
                &\ \int_{\X}\ud \vx\ \rho_t^f(\vx)|\udiv f_t(\vx)| \cdot |\log \bar \rho_t(\vx)| = \int_{\X}\ud \vx\ \bar \rho_0(\vx_0) |\udiv f_t(\vx_t)|\cdot |\log \bar \rho_t(\vx_t)|\\
                \leq&\ \int_{\X}\ud \vx\ \bar \rho_0(\vx_0) L(1+\|\vx_t\|)(1+\|\vx_t\|)^\alpha < \infty.
            \end{align*}
            \begin{align*}
                &\ \int_{\X}\ud \vx\ \rho_t^f(\vx)| [(f_t \cdot \nabla \log \rho_t^f) \log \rho_t](\vx)| = \int_{\X}\ud \vx\ \bar \rho_0(\vx_0)| [(f_t \cdot \nabla \log \rho_t^f) \log \rho_t](\vx_t)|\\
                \leq &\ \int_{\X}\ud \vx\ \bar \rho_0(\vx_0)\|f_t(\vx_t)\|\| \nabla \log \rho_t^f(\vx_t)\| |\log \bar \rho_t(\vx_t)| \\
                \leq&\ \int_{\X}\ud \vx\ \exp(t(1+2L^2))(\|\nabla \log \bar \rho_0(\vx_0)\| + 1)L(1+\|\vx_t\|)(1+\|\vx_t\|)^\alpha < \infty.
            \end{align*}
            \begin{align*}
                &\ \int_{\X}\ud \vx\ \rho_t^f(\vx)| [f_t \cdot \nabla\log \rho_t](\vx)| = \int_{\X}\ud \vx\ \bar \rho_0(\vx_0)\|f_t(\vx_t)\|\|\nabla\log \bar \rho_t(\vx_t)\| \\
                \leq &\ \int_{\X}\ud \vx\ \bar \rho_0(\vx_0)\|f_t(\vx_t)\|\|\nabla\log \bar \rho_t(\vx_t)\| \leq \int_{\X}\ud \vx\ \bar \rho_0(\vx_0)L (1+\|\vx_t\|)(1+\|\vx_t\|)^\alpha < \infty.
            \end{align*}
        \end{itemize}
        \item We handle the term $\rho_t^f \bar f_t$.
        \begin{itemize}
            \item  To show that $\int_{\X}\ud \vx\ \| [\rho_t^f \bar f_t](\vx)\| <\infty$
                \begin{align*}
                    \int_{\X}\ud \vx\ \rho_t^f(\vx)\|\bar f_t(\vx)\| = \int_{\X}\ud \vx\ \bar \rho_0(\vx_0)\|\bar f_t(\vx_t)\| < \infty
                \end{align*}
            \item To show that $\int_{\X}\ud \vx\ | \udiv(\rho_t^f \bar f_t) (\vx)| <\infty$
            \begin{align*}
                \int_{\X}\udiv(\rho_t^f \bar f_t) =&\ \int_{\X} \nabla \rho_t^f \cdot \bar f_t + \rho_t^f  \udiv \bar f_t = \int_{\X} \rho_t^f\left(\nabla \log \rho_t^f \cdot \bar f_t + \udiv \bar f_t \right)\\
                =&\ \int_{\X} \ud \vx_0\ \bar \rho_0(\vx_0)\left(\nabla \log \rho_t^f(\vx_t) \cdot \bar f_t(\vx_t) + \udiv \bar f_t(\vx_t) \right) < \infty,
            \end{align*}
            using the polynomial growth assumption on the ground truth velocity field $\bar f_t$.
        \end{itemize}
    \end{itemize}
\end{proof}

We now focus on addressing the integral-of-divergence term in Lemma \ref{lemma_ModuEnergyEvo}.
The following result will be useful.
\begin{remark}
    Let $X_t^f$ be the flow map generated by the velocity field $f \in \Lip(\sR^d)$. 
    We have that $X_t^f\in \Lip(\sR^d)$ and that $\rho_t^f = X_t^f\sharp\bar \rho_0$ remains bounded for $t \in [0, T]$ if $\bar \rho_0$ is bounded on $\sR^d$. This can be established using the change-of-variable formula of the probability density function. 
\end{remark}

\begin{lemma} \label{lemma_zero_divergence_integral_modulated_energy}
    Under Assumptions \ref{ass_regularity_force_field} to \ref{ass_regularity_ground_truth}, we have
    \begin{equation}\label{eqn_divergence_integral_modulated_energy}
        \int_\X \ud \vx\ \udiv \Big\{ g * (\rho^f_t - \bar \rho_t )(\vx)( \rho^f_t(\vx) f_t(\vx)   - \bar \rho_t(\vx) \bar f_t(\vx))\Big\}  = 0.
    \end{equation}
\end{lemma}
\begin{proof}
    Denote $h = g * (\rho_t - \bar \rho_t )(\vx)( \rho_t(\vx) f_t(\vx)   - \bar \rho_t(\vx) \bar f_t(\vx))$.
    To show that $h \in \gL^1(\X)$, we can show that, after splitting into simple terms, every term from $h$ is in $\gL^1$. In the following, we show 
    \[
        g * \rho_t(\vx) \rho_t(\vx) f_t(\vx) \in \gL^1(\X).
    \] 
    Other terms can be proved similarly.
    First, we show that $g * \rho_t \in \gL^\infty(\X)$ if $\rho_t \in \gL^\infty(\X)$. For any constant $C$, we have
    \begin{align*}
        &\ g * \rho_t(\vx) = \int_\X g(\vx - \vy) \rho_t(\vy) \ud \vy \\
        =&\ \int_{\|\vx - \vy\|\leq C} g(\vx - \vy) \rho_t(\vy) \ud \vy  + \int_{\|\vx - \vy\|> C} g(\vx - \vy) \rho_t(\vy) \ud \vy\\
        \leq&\ \|\rho_t\|_{\gL^\infty(\X)} \int_{\|\vx - \vy\|\leq C} \|\vx - \vy\|^{2-d} \ud \vy + C^{2-d}\int_{\|\vx - \vy\|> C}\rho_t(\vy) \ud \vy \\
        \leq&\ \|\rho_t\|_{\gL^\infty(\X)} C^2 + C^{2-d},
    \end{align*}
    where in the last inequality, we use 
    \begin{equation*}
        \int_{\|\vx - \vy\|\leq C} \|\vx - \vy\|^{2-d} \ud \vy = \int_{\|\vy\|\leq C} \|\vy\|^{2-d} \ud \vy \leq \int_{0\leq r\leq C}r^{2-d} \ud r \int J_\theta \ud_\theta \leq \int_{0\leq r\leq C} \ud r\ r \leq C^2.
    \end{equation*}
    Here $J_\theta$ denotes the determinant of the Jacobian obtained from changing to the polar coordinate, which is bounded by $r^{d-1}$.
    We hence obtain
    \begin{equation*}
        \int_\X \ud \vx\ \|g * \rho_t(\vx) \rho_t(\vx) f_t(\vx)\| \leq C' \int_\X \ud \vx\ \rho_t(\vx) \|f_t(\vx)\| = C'\int_\X \ud \vx_0\ \bar \rho_0(\vx_0) \|f_t(\vx_t)\| < \infty,
    \end{equation*}
    where we use $f_t \in \Lip(\X)$ and the estimation in Lemma \ref{lemma_regularity_estimation}.

    Similarly, to show that $\udiv(h) \in \gL^1(\X)$, we can show that, after splitting into simple terms, every term from $\udiv(h)$ is in $\gL^1$.
    In the following, we show that 
    \[
        \nabla g * \rho_t(\vx) \rho_t(\vx) f_t(\vx) \in \gL^1(\X)\ \text{and}\ g * \rho_t(\vx) \nabla \rho_t(\vx) \cdot f_t(\vx) \in \gL^1(\X).
    \] 
    Other terms can be proved similarly.

    To show that $\nabla g * \rho_t(\vx) \rho_t(\vx) f_t(\vx) \in \gL^1(\X)$, we first show that $\nabla g * \rho_t(\vx) \in \gL^\infty(\X)$ for $\rho_t \in \gL^\infty(\X)$. We can then apply the same argument as above to establish the absolute integrability of the whole term.
    \begin{align*}
        &\ \|\nabla g * \rho_t(\vx)\| \leq \int_\X \|\nabla g(\vx - \vy)\| \rho_t(\vy) \ud \vy \\
        =&\ \int_{\|\vx - \vy\|\leq C} \|\nabla g(\vx - \vy)\| \rho_t(\vy) \ud \vy  + \int_{\|\vx - \vy\|> C} \|\nabla g(\vx - \vy)\| \rho_t(\vy) \ud \vy\\
        \leq&\ \|\rho_t\|_{\gL^\infty(\X)} \int_{\|\vx - \vy\|\leq C} \|\vx - \vy\|^{1-d} \ud \vy + C^{1-d}\int_{\|\vx - \vy\|> C}\rho_t(\vy) \ud \vy \\
        \leq&\ \|\rho_t\|_{\gL^\infty(\X)} C + C^{1-d}.
    \end{align*}

    To show that $g * \rho_t(\vx) \nabla \rho_t(\vx) \cdot f_t(\vx) \in \gL^1(\X)$, we use the fact that $g * \rho_t \in \gL^\infty(\X)$ and that 
    \begin{equation}
       \int_\X \ud \vx\ \nabla \rho_t(\vx) \cdot f_t(\vx) = \int_\X \ud \vx_0\ \bar \rho_0(\vx_0) \nabla \log \rho_t(\vx_t) \cdot f_t(\vx_t).
    \end{equation}
    Using the estimation in Lemma \ref{lemma_regularity_estimation} and that $f_t \in \Lip(\X)$, we obtain the result.
\end{proof}

\begin{proof}[Proof of Lemma \ref{lemma_regularity_estimation}]
    \begin{equation}
        \frac{\ud }{\ud t}\|\vx_t\|^2 \leq \|\vx_t\|^2 + \|\bar f_t(\vx_t)\|^2 \leq \|\vx_t\|^2 + 2L^2(1+\|\vx_t\|^2) = (1 + 2L^2)\|\vx_t\|^2 + 2L^2.
    \end{equation}
    Using Gr\"onwall's inequality, we have
    \begin{equation}
        \|\vx_t\|^2 \leq \exp(t(1+2L^2))(\|\vx_0\|^2 + 2L^2/(1 + 2L^2)) \leq \exp(t(1+2L^2))(\|\vx_0\|^2 + 1).
    \end{equation}

    We have
    \begin{align} \label{eqn_dynamics_of_log_prob}
        \frac{\ud }{\ud t}\log \rho_t^f(\vx_t) = - \udiv \bar f_t(\vx_t)
    \end{align}
    
    We have
    \begin{align}
        \frac{\ud }{\ud t}\nabla \log \rho_t^f(\vx_t) = - \nabla  \left(\udiv \bar f_t(\vx_t) \right) -   \left(\gJ_{\bar f_{t}}(\vx_t)\right)^\top \nabla \log \rho_{t}^{f}(\vx_t)
    \end{align}
    \begin{align}
        \frac{\ud }{\ud t}\|\nabla \log \rho_t^f(\vx_t)\|^2 \leq&\ \|\nabla \log \rho_t^f(\vx_t)\|^2 + 2\|\nabla  \left(\udiv \bar f_t(\vx_t) \right)\|^2 + 2\|\left(\gJ_{\bar f_{t}}(\vx_t)\right)^\top \nabla \log \rho_{t}^{f}(\vx_t)\|^2\\
        \leq &\ \|\nabla \log \rho_t^f(\vx_t)\|^2 (1+2\|\gJ_{\bar f_{t}}(\vx_t)\|^2) + 2\|\nabla  \left(\udiv \bar f_t(\vx_t) \right)\|^2 \\
        \leq &\ \|\nabla \log \rho_t^f(\vx_t)\|^2 (1+2L^2) + 2L^2
    \end{align}
    Using Gr\"onwall's inequality, we have
    \begin{equation}
        \|\nabla \log \rho_t^f(\vx_t)\|^2 \leq  \exp(t(1+2L^2))(\|\nabla \log \bar \rho_0(\vx_0)\|^2 + 1).
    \end{equation}
\end{proof}

\end{document}